\documentclass{amsart}
\allowdisplaybreaks[4]
\usepackage{graphicx}
\usepackage{color}

\newcommand{\R}{\mathbb{R}}

\newcommand{\sig}{\sigma}

\newcommand{{\ba}}{\bf a}
\newcommand{\ve}{\varepsilon}
\newcommand{\la}{\lambda}
\newcommand{\La}{\Lambda}
\newcommand{\ga}{\gamma}

\newcommand{\pa}{\partial}
\newcommand{\ra}{\rightarrow}
\newcommand{\Om}{\Omega}
\newcommand{\del}{\delta}

\newcommand{\al}{\alpha}

\newcommand{\be}{\begin{equation}}
\newcommand{\ee}{\end{equation}}

\newtheorem{lem}{Lemma}{\bf}{\it}
{\it}{\rm}
\newtheorem{rem}{Remark}{\it}{\rm}
{\it}{\rm}

\newtheorem{theorem}{Theorem}
\newtheorem{proposition}{Proposition}

\numberwithin{theorem}{section}
\numberwithin{lem}{section}
\numberwithin{equation}{section}
\numberwithin{proposition}{section}
\numberwithin{corollary}{section}
\title[Carr-Penrose Model]{On Global Asymptotic Stability for the diffusive Carr-Penrose Model }

\author{Joseph G. Conlon and Michael Dabkowski}

\address{(Joseph G. Conlon): University of Michigan\\ Department of Mathematics\\ Ann Arbor,
  MI 48109-1109}
\email{conlon@umich.edu}

\address{(Michael Dabkowski): University of Michigan-Dearborn \\ Department of Mathematics and Statistics\\ Dearborn,
  MI 48128}
\email{mgdabkow@umich.edu}


\begin{document}

\maketitle

\begin{abstract}
This paper is concerned with large time behavior of the solution to a diffusive perturbation of the linear LSW model introduced by Carr and Penrose.  Like the LSW model,  the Carr-Penrose model has a family of rapidly decreasing  self-similar solutions, depending on a parameter $\beta$ with $0<\beta\le 1$.  It is shown that if the initial data has compact support then the solution to the diffusive model at large time approximates the $\beta=1$ self-similar solution. This result supports the intuition that diffusion provides the mechanism whereby the $\beta=1$ self-similar solution of the LSW model is the only physically relevant one.

\end{abstract}

\section{Introduction}
   In this paper we continue our study of the  diffusive Carr-Penrose (CP) model introduced in \cite{cdw}.  The model is obtained by adding a second order diffusion term with coefficient $\ve/2>0$ to the Carr-Penrose equation \cite{cp}. The density function  $c_\ve(x,t)$ evolves according to a linear diffusion equation, subject to the linear mass conservation constraint as follows:
\begin{align}
\frac{\pa c_\ve(x,t)}{\pa t} \ &=  \ \frac{\pa}{\pa x}\left\{\left[1-\frac{x}{\La_\ve(t)}\right]c_\ve(x,t)\right\}+\frac{\ve}{2}\frac{\pa^2 c_\ve(x,t)}{\pa x^2},  \quad x>0, \label{A1}\\
\int_0^\infty x c_\ve(x,t) dx \ &= \  1. \label{B1}
\end{align}
We also need to impose a boundary condition at $x=0$ to ensure that (\ref{A1}), (\ref{B1}) with given initial data $c_\ve(\cdot,0)$, satisfying the constraint (\ref{B1})  has a unique solution.  We impose the Dirichlet boundary condition $c_\ve(0,t)=0, \ t>0$. This condition has the advantage that   the parameter $\La_\ve(t) > 0$ in (\ref{A1}) is  given by a simple formula
\be \label{C1}
\La_\ve(t) = \int^\infty_0 \ x c_\ve(x,t)dx \Big/ \int^\infty_0 c_\ve(x,t) dx        \ = \ 1\Big/ \int^\infty_0 c_\ve(x,t) dx \ .
\ee

In $\S2$ where we discuss the more physically relevant Becker-D\"{o}ring (BD) \cite{bd} and Lifschitz-Slyozov-Wagner (LSW) \cite{ls,w} models, we give some justification for diffusive models with zero Dirichlet condition. 
Penrose \cite{penrose1} argued that solutions of the super-critical  BD model are at large time approximate solutions to the LSW model. This claim was given some rigorous justification by Niethammer \cite{niet}, but much remains to be understood. In particular, the large time behavior of solutions to the LSW model is conjectured to be given by one of a family of self-similar solutions to LSW \cite{cd1}.  These can be conveniently parametrized by a single number $\beta$ with $0<\beta\le 1$. 
Already in the original papers \cite{ls,w} it  was claimed that the only physically relevant self-similar solution to the LSW model is the critical one with $\beta=1$.  Therefore one expects following Penrose \cite{penrose1} that solutions to the super-critical  BD model converge at large time to the critical equilibrium, which forms a boundary layer close to the origin, with the time evolution of the excess mass being well approximated by the $\beta=1$ self-similar solution to the LSW model. 
The only rigorous result known in this direction was obtained in Ball et al \cite{bcp}, where  for the super-critical BD model weak convergence to the critical  equilibrium is established. This implies that the excess mass drifts to infinity at large time, but nothing more precise.  The problem of determining  on what time scale the LSW time evolution becomes a good approximation for the time evolution of the excess mass appears to be extremely subtle. This has been illustrated by some examples  in the paper of Penrose \cite{penrose2}. 

In \cite{cs} we introduced a continuous non-linear Fokker-Planck (NFP) model which bears many similarities to the discrete BD model. In fact it is shown in \cite{cs} that there is a continuous interpolation from the BD to the NFP model. The interpolated models are like BD also discrete, but on the chain $\ve\mathbb{Z}^+$  with $0<\ve\le 1$. In this interpolation the BD model is given by $\ve=1$, and the NFP model by the limit as $\ve\ra0$.  The main result of \cite{cs} is the proof of convergence to equilibrium, so a continuous analogue of the main result of \cite{bcp}. The methodology is also similar, using the fact that a free energy functional exists, which decreases along trajectories of the solution to NFP.  Comparison of the BD system to a continuous diffusive system was first considered by  
Vel\'{a}zquez \cite{vel}. It was argued further in the physics literature \cite{meer,rz} that diffusion may be the mechanism of the selection principle, whereby the $\beta=1$ self-similar solution of LSW describes the asymptotic behavior of the excess mass in the super-critical BD model and some other models of Ostwald ripening  \cite{pego}.

In the NFP model the Fokker-Planck equation is solved with  a non-zero Dirichlet boundary condition, which couples to a parameter in the PDE and conservation law. 
In $\S2$ we observe that one can introduce a second parameter  into the model, which does not affect the Fokker-Planck dynamical law or conservation law, but sends the Dirichlet boundary condition to zero as this parameter goes to infinity. Therefore it is reasonable to expect that the study of the zero Dirichlet boundary condition model yields some insight into
the NFP model. Furthermore, since it is also a diffusive model one expects that the selection principle for the $ \beta=1$ self-similar solution of LSW continues to operate. In the present paper we verify this is the case for the much simpler diffusive CP model (\ref{A1}), (\ref{B1}). 

The system (\ref{A1}), (\ref{B1}) with $\ve\ge 0$  can be interpreted as an evolution  equation for the probability density function (pdf) of random variables. Thus let us assume that the initial data $c_\ve(x,0)\ge 0, \ x>0,$ for (\ref{A1}), (\ref{B1}) satisfies $\int_0^\infty c_\ve(x,0) \ dx<\infty$, and let $X_{\ve,0}$ be the non-negative random variable with  pdf $c_\ve(\cdot,0)/\int_0^\infty c_\ve(x,0) \ dx$. The conservation law (\ref{B1}) implies that the mean $\langle X_{\ve,0}\rangle$ of $X_{\ve,0}$ is finite, and this is the only absolute requirement on the variable $X_{\ve,0}$.  If for $t>0$ the variable $X_{\ve,t}$ has pdf  $c_\ve(\cdot,t)/\int_0^\infty c_\ve(x,t) \ dx$, then (\ref{C1}) implies that  $\La_\ve(t)=\langle  X_{\ve,t}\rangle$, and hence (\ref{A1}) becomes  an evolution equation for the pdf of $X_{\ve,t}$.  

There is an infinite  one-parameter family of self-similar solutions to (\ref{A1}), (\ref{B1}) with $\ve=0$. Using the normalization $\langle X_0\rangle =1$, the initial data for these solutions are given by
\be \label{D1}
P(X_0>x) \ = \ 
\begin{cases}
   [1-(1-\beta)x]^{\beta/(1-\beta)} \ , \ 0<x<1/(1-\beta),  \quad &\text{ if \ } 0<\beta<1, \\
   e^{-x}  &\text{ if \ } \beta=1, \\
 [1+(\beta-1)x]^{\beta/(1-\beta)} \ , \ 0<x<\infty,  \quad  & \ \text{if \ } \beta>1.
\end{cases}
\ee
The random variable $X_t$ corresponding to the evolution (\ref{A1}), (\ref{B1}) with $\ve=0$ and  initial data (\ref{D1}) is then given by
\be \label{E1}
X_t \ = \ \langle X_t\rangle X_0 \ , \quad \frac{d}{dt} \langle X_t\rangle  \ = \beta \ .
\ee
The main result of \cite{cp}  is that a solution of (\ref{A1}), (\ref{B1}) with $\ve=0$ converges at large time to the self-similar solution with parameter $\beta$, provided the initial data  and the self similar solution of parameter $\beta$ behave in the same way at the end of their supports. The main result of this paper is that a similar result holds for the diffusive model (\ref{A1}), (\ref{B1}) with $\ve>0$ and initial data of compact support. In this case a selection principle operates, so that the asymptotic behavior is given by the $\beta=1$ self-similar solution (\ref{D1})  corresponding to the exponential variable:
\begin{theorem}
 Assume the initial data  $c_\ve(\cdot,0)$ for (\ref{A1}), (\ref{B1}) is non-negative, integrable and has compact support. Let $X_{\ve,t}, \ t>0,$ be the random variable corresponding to the solution $c_\ve(\cdot,t)$ of (\ref{A1}), (\ref{B1}) with Dirichlet boundary condition $c_\ve(0,t)=0$.  Then
 \be \label{F1}
 \frac{X_{\ve,t}}{\langle X_{\ve,t}\rangle} \xrightarrow{D} \mathcal{X} \ \ {\rm as \ \ } t\ra\infty \ , \quad 
 \lim_{t\ra\infty} \frac{d}{dt}  \langle X_{\ve,t}\rangle \ = \ 1 \ ,
 \ee
 where $\mathcal{X}$ is the exponential random variable with mean $1$, and $\xrightarrow{D}$ denotes convergence in distribution. 
\end{theorem}
A result analogous to Theorem 1.1 was proved in \cite{cdw} for a reduced model which we denoted  the {\it inviscid} CP model, since a corresponding viscous CP model  is equivalent to the diffusive CP model (\ref{A1}), (\ref{B1}).  The evolution of the inviscid model
is determined by solving a first order non-linear PDE. This PDE is essentially the $\ve=0$ CP equation (\ref{A1}) with the addition of a quadratic non-linearity. The method of characteristics may be used to solve this Burgers' type equation in the case of no shocks. Hence  no boundary condition is needed at $x=0$ to guarantee uniqueness of the solution.  In $\S3$ we introduce a diffusive model,  also without boundary condition at $x=0$, by solving (\ref{A1}) on the whole line 
$-\infty<x<\infty$ with initial data which is zero on the half line $x<0$ and non-negative for $x>0$. The parameter $\La_\ve(t)$ in (\ref{A1}) is given by the first formula on the RHS of (\ref{C1}). The conservation law (\ref{B1}) now no longer holds.  We prove
a result analogous to Theorem 1.1 for this model.

The remainder of the paper consists of generalizing the results of the whole line diffusive CP model to the half line model  with zero Dirichlet condition at $x=0$. 
One can most easily see how this introduces additional subtlety  into the problem, by comparing the formulas for the rates of coarsening in the $\ve=0$ CP and
diffusive $\ve>0$ models.  On differentiating (\ref{C1}) we see that  if $\ve=0$ then
\be \label{G1}
\frac{d\La_0(t)}{dt} \ = \ c_0(0,t)\Big/\left[ \int^\infty_0 c_0(x,t) dx\right]^2 \ ,
\ee
whereas if $\ve>0$ the formula is given by
\be \label{H1}
\frac{d\La_\ve(t)}{dt} \ = \ \frac{\ve}{2}\frac{\pa c_\ve(0,t)}{\pa x}\Big/\left[ \int^\infty_0 c_\ve(x,t) dx\right]^2.
\ee
It  clearly follows from (\ref{G1}) that the function $t\ra \La_0(t)$ is increasing. We also see from (\ref{H1}) and the maximum principle \cite{pw} applied to (\ref{A1}) that the function $t\ra\La_\ve(t)$ is strictly  increasing if $\ve>0$. 

In comparing the CP to the diffusive CP model, an obvious question to ask is if the limit of the solution to the diffusive model on a fixed time interval $0<t\le T$, with parameter $\ve>0$  and given initial data independent of $\ve$, converges as $\ve\ra0$ to the solution of the CP model with the same initial data.  A strong convergence result was proven in \cite{cdw}, and a similar result for a diffusive LSW model in \cite{c1}.  In $\S2$ we explain how the main theorem of Niethammer \cite{niet}  may be interpreted as an analogous result for the BD model.  In the proof of convergence for the diffusive CP and LSW models a boundary layer analysis is necessary. We can see this from (\ref{G1}), (\ref{H1}) since these equations suggest that
\be \label{I1}
\lim_{\ve\ra 0} \frac{\ve}{2}\frac{\pa c_\ve(0,t)}{\pa x} \ = \ c_0(0,t) \ .
\ee
From (\ref{I1}) we expect  there is a boundary layer of size $O(\ve)$ at the origin, within which the function $x\ra c_\ve(x,t)$ increases rapidly from $0$ to $c_0(0,t)$. 
The analysis of this boundary layer is the main source of difficulty in proving convergence.

Theorem 1.1 concerns large time behavior of solutions to the diffusive problem. That is $\ve>0$ is fixed, and we are interested in the behavior of solutions $c_\ve(\cdot,t)$ to (\ref{A1}), (\ref{B1}) with Dirichlet condition $c_\ve(0,\cdot)=0$ as $t\ra\infty$. There is a close relation between this problem and the problem of the convergence of $c_\ve(\cdot,t)$ as $\ve\ra 0$ on a fixed time interval $0<t\le T$. The reason is that $\lim_{t\ra\infty}\La_\ve(t)=\infty$.  Since the $\ve=0$ CP model is scale invariant, rescaling of the PDE (\ref{A1}) so that the mean mass at large time $T$ is $O(1)$ makes the diffusion constant  in (\ref{A1}) small while leaving the other terms the same. Therefore large time behavior $T\ra\infty$ of $c_\ve(\cdot,T)$   may be estimated by means of $\ve\ra 0$ analysis, {\it provided} we can establish certain uniformity properties on the solution as $T\ra\infty$. 

Uniformity is measured in terms of the boundedness of the {\it beta} function of a random variable, which is defined by (\ref{W2}), (\ref{X2}).  There is a close relationship between boundedness properties of the beta function and log concavity properties of the pdf.  In \cite{cdw} we showed that if the initial data $c_\ve(\cdot,0)$  for (\ref{A1}) with zero Dirichlet condition satisfies a log concavity property then  $c_\ve(\cdot,t)$ has the same log concavity property for all $t>0$.
Equivalently, if the beta function  for $X_{\ve,0}$ is bounded by $1$ then the beta function for $X_{\ve,t}$ is bounded by $1$ for all $t>0$.  Since the coarsening rate $d\La_\ve(t)/dt$ can be estimated in terms of the beta function of $X_{\ve,t}$, this implies a uniform upper bound on the rate of coarsening as $t\ra\infty$.  

The main technical issues of this paper are concerned with the ratio (\ref{A3}) of the Dirichlet Green's function for (\ref{A1}) on the half line to the whole line Green's function.  The whole line Green's function is explicitly given by the Gaussian (\ref{D2}). In \cite{cd} it is shown that this ratio satisfies a log concavity condition. Using this fact and other estimates from \cite{cd}  we are able to prove that if the initial data for (\ref{A1}) has compact support then the beta function of $X_{\ve,T}$ is bounded at large $T$, and furthermore converges to the beta function of the exponential variable, whose beta function is simply the constant $1$. Convergence of the coarsening rate as in (\ref{F1}) then follows from an $\ve\ra 0$ analysis. 

\vspace{.1in}

\section{Relation to the BD, NFP and LSW models}
The NFP model introduced in \cite{cs} is a continuous version of the BD model in which the dynamics is given by a Fokker-Planck (FP) equation on the positive real line with a time varying parameter. The parameter is coupled to a Dirichlet boundary condition and a conservation law. The  FP equation is given by
\be \label{A6}
\pa_tc(x,t)+\pa_x(b(x,t)c(x,t)) \ = \ \pa^2_x(a(x)c(x,t)) \ , \quad x,t>0 \ ,
\ee
where the drift $b(\cdot,\cdot)$ is of the form
\be \label{B6}
b(x,t) \ = \ a(x)[\theta(t)W'(x)-V'(x)] \ .
\ee
The functions $V(\cdot), \ W(\cdot)$ are assumed to be  positive $C^1$ and increasing.   The conservation law is given by
\be \label{C6}
A\theta(t)+\int_0^\infty W(x) c(x,t) \ dx \ = \ \rho \ , \quad {\rm where \ } A\ge 0  \ {\rm is \ constant}.
\ee
It is easy to see that $c(x,t)=c_\theta^{\rm eq}(x)$ where
\be \label{D6}
c_\theta^{\rm eq}(x) \ = \ a(x)^{-1}\exp(-V(x)+\theta W(x)) \ ,
\ee
is an equilibrium solution to (\ref{A6}), (\ref{B6}). 
The Dirichlet boundary condition is given in terms of this equilibrium solution  by
\be \label{E6}
c(0,t) \ = \ c_{\theta(t)}^{\rm eq}(0) \ , \quad t>0 \ .
\ee
The NFP model consists of solving (\ref{A6}) subject to the constraints (\ref{C6}), (\ref{E6})  with given non-negative initial data. 

Observe that if we replace the potential $V(\cdot)$ with $V(\cdot)+M$ where $M$ is a constant, the PDE (\ref{A6}) and conservation law (\ref{C6}) do not change, just the boundary condition (\ref{E6}).  If we let  $M\ra\infty$ then (\ref{E6}) becomes the zero Dirichlet condition in the limit. Thus we have a continuous interpolation between the NFP model and a diffusive model with zero Dirichlet condition.

One can see from the formulation of the BD model in \cite{c1}, that if $a(\cdot), V(\cdot), W(\cdot)$ are given by the formulae
\be \label{F6}
a(x)= (1+x)^\al, \ 0<\al<1, \quad W(x)=1+x, \quad a(x)V'(x)=1, \ V(0)=0,
\ee
then the NFP model can be considered a continuous version of the BD model.  In the case of (\ref{F6}) equilibrium solutions (\ref{D6})  of the NFP model, 
which satisfy the conservation law (\ref{C6}), exist for all $\theta\le 0$.  The $\theta=0$ equilibrium solution is the {\it critical} equilibrium. Let $\rho_{\rm crit}$ be the value of $\rho$ in (\ref{C6}) which corresponds to the critical equilibrium, so
\be \label{G6}
\int_0^\infty W(x) c_0^{\rm eq}(x) \ dx \ = \ \rho_{\rm crit} \ .
\ee
In \cite{cs} it is shown that for $\rho>\rho_{\rm crit}$ in (\ref{C6}) the solution $c_\ve(\cdot,t)$ of (\ref{A6})-(\ref{E6}) converges weakly as $t\ra\infty$  to 
$c_0^{\rm eq}(\cdot)$. The main tool in the proof is the free energy functional $\mathcal{G}$ defined by
\be \label{H6}
\mathcal{G}(c(\cdot),\theta) \ = \ \int_0^\infty [\log \frac{c(x)}{c^{\rm eq}_0(x)}-1]c(x) \ dx+ \frac{A\theta^2}{2} \ ,
\ee
which decreases along orbits $t\ra c(\cdot,t)$ of the solution to (\ref{A6})-(\ref{E6}). One can easily see this by writing (\ref{A6}) as
\be \label{I6}
\pa_tc(x,t) \ = \ \pa_x\left(a(x)c(x,t)\pa_x\log\frac{c(x,t)}{ c_{\theta(t)}^{\rm eq}(x)}\right)  \ .
\ee
Letting $\mathcal{D}$ be the dissipation functional
\be \label{J6}
\mathcal{D}(c(\cdot),\theta) \ = \ \int_0^\infty  a(x)\left(\pa_x\log\frac{c(x)}{c^{\rm eq}_\theta(x)}\right)^2 c(x) \ dx \ ,
\ee
we see from (\ref{C6}), (\ref{E6}) and (\ref{H6}), (\ref{I6}) that
\be \label{K6}
\frac{d}{dt}\mathcal{G}(c(\cdot,t),\theta(t)) \ = \ -\mathcal{D}(c(\cdot,t),\theta(t))  \ .
\ee

Recently an alternative continuous model has been proposed by Goudon and Monasse \cite{gm}.  This model is quite similar to the NFP model, but the time varying parameter which occurs in the conservation law and boundary condition also enters the second derivative term in the Fokker-Planck PDE, in addition to the first derivative term as in (\ref{A6}), (\ref{B6}).  The parameter in \cite{gm} seems to represent mass concentrated at the origin, which is similar to the situation with the BD model. 
However there does not appear to be a free energy functional which decreases along orbits of the solution.  Therefore a proof of convergence at large time  to equilibrium is not available by Lyapounov's second method as was the case in \cite{bcp, cs}.  

A more closely related model is that studied by Niethammer and Pego \cite{np1}.   This is a generalized LSW model, which  can be obtained from the NFP model by eliminating the second derivative term in (\ref{A6}), yielding a first order PDE. The boundary condition (\ref{E6}) is no longer required, but the conservation law remains as in (\ref{C6}).  Global existence and uniqueness of solutions are proved, as well as continuous dependence on the initial data.  An important tool is the energy functional $\mathcal{E}$ defined by 
\be \label{L6}
\mathcal{E}(c(\cdot),\theta) \ = \ \int_0^\infty V(x)c(x) \ dx +\frac{A\theta^2}{2} \ .
\ee
Defining the dissipation functional by
\be \label{M6}
\mathcal{D}(c(\cdot),\theta) \ = \  \int_0^\infty a(x) [\theta W'(x)-V'(x)]^2c(x) \ dx  \  ,
\ee
we see that along orbits $t\ra c(\cdot,t)$ of the solution one has
\be \label{N6}
\frac{d}{dt}\mathcal{E}(c(\cdot,t),\theta(t)) \ = \ -\mathcal{D}(c(\cdot,t),\theta(t))  \ ,
\ee
provided $V(0)=W(0)=0$. 

We next  introduce a small parameter $\ve>0$ into the NFP model  with the goal of proving convergence on a fixed time scale to the LSW model as  $\ve\ra0$. Since we want this to be related to the problem of large time behavior of the NFP model,  we scale the solution $c(\cdot,\cdot)$ of (\ref{A6})-(\ref{E6})  by  setting $c_\ve(x,t)=\ga(\ve)^{-1}c(x/\ve,t/\ve)$, where  $\ga(\cdot)$ is a positive function. The PDE (\ref{A6}), (\ref{B6})  then becomes the PDE
\be \label{O6}
\pa_tc_\ve(x,t)+\pa_x(b_\ve(x,t)c_\ve(x,t)) \ = \ \pa^2_x(a_\ve(x)c_\ve(x,t)) \ , \quad x,t>0 \ ,
\ee
where
\begin{multline} \label{P6}
a_\ve(x)=\ve a(x/\ve), \quad b_\ve(x,t)= b(x/\ve,t/\ve)=a_\ve(x)[\theta_\ve(t)W'_\ve(x)-V'_\ve(x)] \ , \\
V_\ve(x)= V(x/\ve), \quad \theta_\ve(t)W_\ve(x)=\theta(t/\ve)W(x/\ve)  \ .
\end{multline}
This defines $\theta_\ve(\cdot), \ W_\ve(\cdot)$ up to a multiplicative constant. The constant is determined from the scaled conservation law
\be \label{Q6}
A_\ve\theta_\ve(t)+\int_0^\infty W_\ve(x) c_\ve(x,t) \ dx \ = \ \rho \ .
\ee
Thus we have that
\be \label{R6}
 A_\ve=\ve^{-1}\ga(\ve)A, \quad \theta_\ve(t)=\ve\ga(\ve)^{-1}\theta(t/\ve), \quad W_\ve(x)= \ve^{-1}\ga(\ve)W(x/\ve) \ .
\ee
The boundary condition (\ref{E6}) becomes
\be \label{S6}
c_\ve(0,t) \ = \ [\ga(\ve) a(0)]^{-1}\exp\left[-V(0)+\theta(t/\ve)W(0)\right] \ = \ c^{\rm eq}_{\theta_\ve(t),\ve}(0) \ ,
\ee
where the equilibrium density $c^{\rm eq}_{\theta,\ve}(\cdot)$ is given by 
\be \label{T6}
c_{\theta,\ve}^{\rm eq}(x) \ = \ [\ga(\ve)\ve^{-1}a_\ve(x)]^{-1}\exp(-V_\ve(x)+\theta W_\ve(x)) \ .
\ee
Note that
\be \label{U6}
\int_0^\infty W_\ve(x)c_{0,\ve}^{\rm eq}(x) \ dx \ = \ \rho_{\rm crit} \ ,
\ee
where $\rho_{\rm crit}$ is given by (\ref{G6}). 
We see from (\ref{H6}) that the  free energy $\mathcal{G}_\ve$  corresponding to (\ref{O6})-(\ref{T6}) is given by
\be \label{V6}
\mathcal{G}_\ve(c(\cdot),\theta) \ = \ \int_0^\infty [\log \frac{c(x)}{c^{\rm eq}_{0,\ve}(x)}-1]c(x) \ dx+ \frac{A_\ve\theta^2}{2} \ .
\ee
Note from (\ref{H6}), (\ref{V6}) that
\be \label{W6}
\mathcal{G}_\ve(c_\ve(\cdot,t),\theta_\ve(t)) \ = \ \ve\ga(\ve)^{-1}\mathcal{G}(c(\cdot,t/\ve),\theta(t/\ve)) \ .
\ee

We consider now the situation where the coefficients are given by (\ref{F6}).  We choose $\ga(\ve)$ such that $\lim_{\ve\ra 0} W_\ve(\cdot)=W_0(\cdot)$ exists.  Evidently we should take $\ga(\ve)=\ve^2$, in which case $W_0(x)=x$.  Next we need that $\lim_{\ve\ra0} a_\ve(x)\theta_\ve(t)W'_\ve(x)$ exists. Defining $\tilde{\theta}_\ve(t)= \ve^{1-\al}\theta_\ve(t)$, we see the limit exists if $\lim_{\ve\ra0}\tilde{\theta}_\ve(t)=\tilde{\theta}_0(t)$ exists. The conservation law (\ref{Q6}) then becomes
\be \label{X6}
A\ve^\al\tilde{\theta}_\ve(t)+\int_0^\infty W_\ve(x) c_\ve(x,t) \ dx \ = \ \rho \ .
\ee
 Letting $\ve\ra0$ and assuming $\lim_{\ve\ra0}c_\ve(x,t)=c_0(x,t), \ \lim_{\ve\ra0}\tilde{\theta}_\ve(t)=\tilde{\theta}_0(t)$, we see that $c_0(\cdot,\cdot)$ is formally a solution to the LSW equation
\be \label{Y6}
\frac{\pa c_0(x,t)}{\pa t}+\frac{\pa}{\pa x}\left[\theta_0(t) x^\al-1\right] \ = \ 0 \ .
\ee
The formal limit of the conservation law (\ref{X6}) yields
\be \label{Z6}
\int_0^\infty W_0(x) c_0(x,t) \ dx \ = \ \rho \ .
\ee

The issue of the limiting behavior as $\ve\ra 0$ of solutions $c_\ve$ to (\ref{O6})-(\ref{T6}) with $a(\cdot), V(\cdot), W(\cdot)$ as in (\ref{F6}) is actually more subtle than indicated due to the fact that the boundary condition $c_\ve(0,t)=c^{\rm eq}_{\theta_\ve(t),\ve}(0)$ tends to force $c_\ve$ to blow up as $\ve\ra 0$ since it is given by the formula
\be \label{AA6}
c_\ve(0,t) \ = \ \ve^{-2}\exp[\ve^\al\tilde{\theta}_\ve(t)] \ .
\ee
This happens unless the initial data within the boundary layer $x= O(\ve)$ is close to the equilibrium $c^{\rm eq}_{0,\ve}(\cdot)$, hence depending on $\ve$. Outside the boundary layer it may be taken independent of $\ve$. The closeness to equilibrium inside the boundary layer then yields from (\ref{U6}), (\ref{X6})  the limiting conservation law 
\be \label{AB6}
\int_0^\infty W_0(x) c_0(x,t) \ dx \ = \ \rho-\rho_{\rm crit} \ ,
\ee
 instead of (\ref{Z6}).

Niethammer \cite{niet} (and Schlichting \cite{sc} by a different method)  has rigorously shown in the context of the Becker-D\"{o}ring model that a weak limit $\lim_{\ve\ra0}c_\ve$ does exist and is a solution to the LSW model (\ref{Y6}), (\ref{AB6}). Closeness of the initial data to equilibrium within the boundary layer is measured in terms of the order of magnitude of the excess free energy 
$\mathcal{G}_\ve(c_\ve(\cdot,0),0)-\mathcal{G}_\ve(c^{\rm eq}_{0,\ve},0)$ as $\ve\ra 0$. 
With $a(\cdot), V(\cdot), W(\cdot)$ as in (\ref{F6}) we have that
\begin{multline} \label{AC6}
\ve^{1-\al}\mathcal{G}_\ve(c_\ve(\cdot,t),\theta_\ve(t)) \ = \ \int_0^\infty\left\{\frac{(\ve+x)^{1-\al}-\ve^{1-\al}}{1-\al}\right\} c_\ve(x,t) \ dx \\
+\ve^{1-\al}\left[\int_0^\infty [\log c_\ve(x,t)-1]c_\ve(x,t) \ dx+\int_0^\infty\log\{\ve^{2-\al}(\ve+x)^\al\} c_\ve(x,t) \ dx \right] + \frac{1}{2}\ve^\al A\tilde{\theta}_\ve(t)^2 \ .
\end{multline}
If $\lim_{\ve\ra 0}c_\ve=c_0$ exists then the first term on the RHS of (\ref{AC6}) converges as $\ve\ra 0$  to the LSW energy $(1-\al)^{-1}\int_0^\infty x^{1-\al} c_0(x,t) \ dx$ corresponding to (\ref{L6}).    Niethammer's main condition is that  $\ve^{1-\al}$ times the excess free energy is bounded independent of $\ve>0$. It is quite easy to come up with a large class of initial data which satisfy the Niethammer condition. Thus suppose $x\ra c_0(x,0), \ x\ge0,$ is a continuous non-negative integrable function which satisfies (\ref{AB6}) with $t=0$.  We define initial data for the system (\ref{O6})-(\ref{T6}) as
\begin{multline} \label{AD6}
c_\ve(x,0) \ = \ c_{0,\ve}^{\rm eq}(x),   \    {\rm if \ } 0\le x<\ve\left[M\log(1/\ve)\right]^{1/(1-\al)}, \\
 c_\ve(x,0) \ = \ c_0(x,0),  \ {\rm if} \ x\ge \ve\left[M\log(1/\ve)\right]^{1/(1-\al} \ ,
\end{multline}
where the constant $M$ satisfies $M>1$. 
For the initial data (\ref{AD6}) it follows from (\ref{U6}), (\ref{X6}), (\ref{AB6}) at $t=0$ that $\tilde{\theta}_\ve(0)=O[\ve^{1-\al}]$. Thus the final term on the RHS of (\ref{AC6}) is $O[\ve^{2-\al}]$ as $\ve\ra0$ and hence bounded independent of $\ve$. If we subtract from the middle term the corresponding quantity with $c_\ve(\cdot,0)$ replaced by $c_{0,\ve}^{\rm eq}(\cdot)$ we obtain an integral over the interval $\ve\left[M\log(1/\ve)\right]^{1/(1-\al)}<x<\infty$. Evidently this difference is $O[\ve^{1-\al}]$, whence again uniformly bounded as $\ve\ra 0$. Finally if we subtract   from the first term the corresponding quantity with $c_\ve(\cdot,0)$ replaced by $c_{0,\ve}^{\rm eq}(\cdot)$ we obtain the LSW energy for $c_0(\cdot,0)$ plus $O[\ve^{1-\al}]$.  Assuming Niethammer's method applies to the present situation,  one expects it to imply that the solution $c_\ve$ to  (\ref{O6})-(\ref{T6}) with initial data (\ref{AD6}) converges weakly to the corresponding solution of the LSW model (\ref{Y6}), (\ref{AB6}). 

So far there are no proofs of strong convergence of solutions to (\ref{O6})-(\ref{T6}) with initial data (\ref{AD6}). However strong convergence for a closely related problem (with $\al=1/3$) was proven in \cite{c1}.  The important difference between the system studied in \cite{c1} and the present one lies in the boundary condition. The diverging Dirichlet condition (\ref{AA6}) is replaced by the zero one $c_\ve(0,t)=0$. There are some other minor differences: in \cite{c1} the conservation law  differs from (\ref{Q6}) in taking $A_\ve=0$ and replacing $W_\ve$ by $W_0$; $\ve^{\al-1}a_\ve(x)W'_\ve(x)=(\ve+x)^\al$ is replaced by $x^\al$. 

As far as understanding large time behavior of the system (\ref{A6})-(\ref{E6}), the replacement of the Dirichlet condition (\ref{E6}) by the zero Dirichlet condition should simplify the problem considerably. One also expects that  good understanding of the zero Dirichlet case will yield significant insight into the 
non-zero Dirichlet case.  Another 
important property of the zero Dirichlet boundary condition model is that the Kohn-Otto argument \cite{ko} may be applied to yield a global weak upper bound on the rate of coarsening. This has been shown for the model in \cite{c1}. However the argument is delicate and does not easily extend  to the slightly modified model  with coefficients (\ref{F6}) and conservation law (\ref{C6})  considered here. 

\vspace{.1in}

\section{The whole line problem}
We consider here the problem of solving the PDE (\ref{A1})  on the whole line $\R=\{-\infty<x<\infty\}$ with initial data which is non-negative on the positive half line $\{x>0\}$ and zero on the negative half line $\{x<0\}$. The parameter function $\La_\ve(\cdot)$ is given by (\ref{C1}), whence the problem is non-linear.  We have now instead of (\ref{D1}) the formula
\be \label{A2}
\frac{1}{\La_\ve(t)}\frac{d\La_\ve(t)}{dt} \ = \ \left\{c_\ve(0,t)+\frac{\ve}{2}\left[\frac{\pa c_\ve(0,t)}{\pa x}+\frac{c_\ve(0,t)}{\La_\ve(t)}\right]\right\} \Big/ \int^\infty_0 c_\ve(x,t) dx \ .
\ee
We shall see below that  (\ref{A2}) implies  $\La_\ve(\cdot)$ is  an increasing function, although $\pa c_\ve(0,t)/\pa x$ may be negative. 

There is an explicit formula for the whole line Green's function for (\ref{A1}).  Let $A:[0,\infty)\ra\mathbb{R}$ be a continuous function and define related functions
$m_{1,A},m_{2,A},\sig^2_A:[0,\infty)\ra\mathbb{R}^+$ by 
\begin{multline} \label{B2}
m_{1,A}(T)= \exp\left[\int_0^T A(s') ds'\right], \quad 
m_{2,A}(T)= \int_0^T \exp\left[\int_{s}^T A(s') ds'\right]  \  ds \ ,  \\
\sig_A^2(T)= \int_0^T \exp\left[2\int_{s}^T A(s') ds'\right] ds \ .  
\end{multline}
The solution to (\ref{A1}) on the whole line with initial data  $c_\ve(x,0), \ x\in\mathbb{R}$, which is supported on the half line $\mathbb{R}^+$,  has the representation 
\be \label{C2}
c_\ve(x,T) \ = \ \int_0^\infty G_\ve(x,y,0,T) c_\ve(y,0) \ dy \ , \quad x\in\mathbb{R}, \ T>0. 
\ee
The Green's function $G_\ve(x,y,0,T)$ is defined (see  eqn. (2.5) of \cite{cd})  by the formula
\be \label{D2}
G_\ve(x,y,0,T)=\frac{1}{\sqrt{2\pi\ve\sig_A^2(T)}}\exp\left[-\frac{\{x+m_{2,A}(T)-m_{1,A}(T)y\}^2}{2\ve\sig_A^2(T)}\right] \ ,
\ee
where $m_{1,A},m_{2,A},\sig^2_A$ are as in (\ref{B2}) and $A(\cdot)\equiv1/\La_\ve(\cdot)$. 
Observe now that
\be \label{E2}
\left[1+\frac{\ve}{2}\frac{\pa}{\pa x}\right]G_\ve(x,y,0,T)  \Big|_{x=0}\ = \ \left[1-\frac{m_{2,A}(T)}{2\sig^2_A(T)}+\frac{m_{1,A}(T)y}{2\sig^2_A(T)}\right]G_\ve(0,y,0,T) \ .
\ee
The RHS of (\ref{E2}) is non-negative for $y>0$ if $A(\cdot)\ge 0$, and consequently from (\ref{C2})  the RHS of  (\ref{A2}) is also non-negative. We have shown that  the function $\La_\ve(\cdot)$ is increasing.  
\begin{lem}
Assume $c_\ve(x,0), \ x\in\R,$  is a non-negative function satisfying $c_\ve(x,0)=0$ for $x<0$, and the integrability condition
\be \label{F2}
0 \ < \ \int_0^\infty(1+x)c_\ve(x,0) \ dx \ < \ \infty \ .
\ee
Let $c_\ve(\cdot,t), \ t>0,$ be the solution to the whole line PDE (\ref{A1}) with the constraint (\ref{C1}) and initial data $c_\ve(\cdot,0)$.  Then the function $\La_\ve(\cdot)$ is increasing and $\lim_{T\ra\infty}\La_\ve(T)=\infty$. If in addition $c_\ve(\cdot,0)$ has compact support, then $\lim_{T\ra\infty}m_{1,A}(T)=\infty$, $\lim_{T\ra\infty} \frac{m_{2,A}(T)}{m_{1,A}(T)}=\infty$,  $\lim_{T\ra\infty} \frac{\sig^2_A(T)}{m_{2,A}(T)}=\infty$, and $\lim_{T\ra\infty} \frac{\sig^2_A(T)}{m_{1,A}(T)^2}<\infty$, where $A(\cdot)\equiv1/\La_\ve(\cdot)$. 
\end{lem}
\begin{proof}
Since $\La_\ve(\cdot)$ is increasing we have that $\lim_{T\ra\infty}\La_\ve(T)=\La_\infty\le \infty$.  We assume $\La_\infty<\infty$ and obtain a contradiction.  With $A(\cdot)=1/\La_\ve(\cdot)$, we see from  (\ref{B2})  that the functions  $m_{1,A}, \ m _{2,A}, \ \sig_A^2$ satisfy
\be \label{G2}
\lim_{T\ra\infty}m_{1,A}(T)=\infty \ , \quad \La_\ve(0)\le\lim_{T\ra\infty} \frac{m_{2,A}(T)}{m_{1,A}(T)} \le \ \La_\infty \ ,
 \quad \frac{\La_\ve(0)}{2}\le\lim_{T\ra\infty} \frac{\sig^2_A(T)}{m_{1,A}(T)^2} \le \frac{\La_\infty}{2}  \ .
\ee
We can obtain a formula for the RHS of (\ref{C1}) in terms of expectations of the random variable  $X_{\ve,y,T}$ given by 
\be \label{H2}
X_{\ve,y,T} \ = \ m_{1,A}(T)y-m_{2,A}(T)+\sqrt{\ve}\sig_A(T)Z \ ,
\ee
where $Z$ is the standard normal variable. 
This follows by observing that
\begin{eqnarray} \label{I2}
\int_0^\infty c_\ve(x,T) \ dx \ &=& \ \int_0^\infty P(X_{\ve,y,T}>0) c_\ve(y,0) \ dy \ , \\
\int_0^\infty xc_\ve(x,T) \ dx \ &=& \ \int_0^\infty E[X_{\ve,y,T} \ | \ X_{\ve,y,T}>0]P(X_{\ve,y,T}>0) c_\ve(y,0) \ dy \ .  \nonumber
\end{eqnarray}

We have now that
\be \label{J2}
 E[X_{\ve,y,T} \ | \ X_{\ve,y,T}>0] \ = \ \sqrt{\ve}\sig_A(T)E[Z-z_{y,T}/\sqrt{\ve} \ | \ Z>z_{y,T}/\sqrt{\ve}] \ , \quad z_{y,T}=\frac{m_{2,A}(T)-m_{1,A}(T)y}{\sig_A(T)} \ .
\ee
Writing
\be \label{K2}
z_{y,T} \ = \ a(T)-b(T)y \ ,
\ee
we see from (\ref{G2}) that
\be \label{L2}
0<\inf_{T\ge 1} a(\cdot)\le \sup _{T\ge 1}a(\cdot)<\infty \ , \quad 0<\inf_{T\ge 1} b(\cdot)\le \sup_{T\ge 1} b(\cdot)<\infty \ .
\ee
We also have from Lemma A.1 there is a constant $c_\ve>0$ such that
\be \label{M2}
E[Z-z \ | \ Z>z] \ \ge \ \max\{c_\ve,-z\} \quad {\rm for \ } z\le \sup_{T\ge1} a(\cdot)/\sqrt{\ve} \ .
\ee
It follows from (\ref{C1}) and (\ref{I2})-(\ref{M2}) that
\be \label{N2}
\La_\ve(T) \ \ge \ c_\ve\sqrt{\ve}\sig_A(T) \ .
\ee
Since (\ref{G2}) implies that $\lim_{T\ra\infty}\sig_A(T)=\infty$, we conclude from (\ref{N2}) that $\lim_{T\ra\infty}\La_\ve(T)=\infty$, contradicting our assumption that $\La_\infty<\infty$. 

Having shown that $\lim_{T\ra\infty}\La_\ve(T)=\infty$, we next show that $\lim_{T\ra\infty}m_{1,A}(T)=\infty$. This implies that $\La_\ve(T)$ cannot grow too rapidly with $T$.  Since $m_{1,A}(\cdot)$ is an increasing function we have that  $\lim_{T\ra\infty}m_{1,A}(T)=m_{1,A}(\infty)\le \infty$. Arguing again by contradiction  we assume that
$m_{1,A}(\infty)<\infty$.  Using the identities
\be \label{O2}
 \frac{m_{2,A}(T)}{m_{1,A}(T)}  \ = \ \int_0^T \frac{dt}{m_{1,A}(t)} \ , \quad \frac{\sig^2_A(T)}{m_{1,A}(T)^2}  \ = \ 
  \int_0^T \frac{dt}{m_{1,A}(t)^2} \ ,
\ee
we see that the functions $a(\cdot), \ b(\cdot)$ of (\ref{K2}) satisfy inequalities
\be \label{P2}
c_1\sqrt{T} \ \le a(T) \ \le C_1\sqrt{T} \ , c_1 \ \le \sqrt{T}b(T) \ \le C_1 \quad {\rm for \ } T>0 \ ,
\ee
where $c_1,C_1>0$ are constants.  Assume now that $c_\ve(\cdot,0)$ has support in the interval $[0,y_\infty]$. It follows from   (\ref{J2}), (\ref{P2}) and Lemma A.1 there are constants $C_2,T_2>0$ such that
\be \label{Q2}
 E[X_{\ve,y,T} \ | \ X_{\ve,y,T}>0] \ \le \  C_2 \quad {\rm for \ } T\ge T_2, \ y\in[0,y_\infty] \ .
\ee
From (\ref{C1}), (\ref{I2})  we conclude that $\lim_{T\ra\infty}\La_\ve(T)<\infty$, a contradiction. 

Next  we show that  $\lim_{T\ra\infty}m_{2,A}(T)/m_{1,A}(T)=\infty$. Note from (\ref{O2}) this implies $m_{1,A}(T)$ cannot approach $\infty$ too rapidly, which in turn implies a lower bound on the rate of growth of $\La_\ve(T)$.   To see this we assume for contradiction that  $\lim_{T\ra\infty}m_{2,A}(T)/m_{1,A}(T)= m_\infty<\infty$, whence (\ref{O2}) and the inequality $m_{1,A}(\cdot)\ge 1$ implies that $\lim_{T\ra\infty}\sig^2_A(T)/m_{1,A}(T)^2\le m_\infty<\infty$. It follows that the functions $a(\cdot), b(\cdot)$ of (\ref{K2}) satisfy the inequality (\ref{L2}).  Hence $z_{y,T}$ is uniformly bounded for $y\in[0,y_\infty]$ as $T\ra\infty$.  We conclude from 
(\ref{J2})-(\ref{L2}) there are positive constants $c_3,T_3$ such that
\be \label{R2}
 E[X_{\ve,y,T} \ | \ X_{\ve,y,T}>0] \ \ge \  c_3m_{1,A}(T) \quad {\rm for \ } T\ge T_3, \ y\in[0,y_\infty] \ .
\ee
It follows from (\ref{C1}), (\ref{I2}), (\ref{R2}) that
\be \label{S2}
\La_\ve(T) \ \ge \ c_3m_{1,A}(T) \quad \ {\rm for \ } T\ge T_3 \ . 
\ee
From (\ref{O2}), (\ref{S2}) there is a constant $C_4>0$ such that
\be \label{T2}
\log m_{1,A}(T) \ = \ \int_0^T\frac{dt}{\La_\ve(t)} \ \le \ C_4 \frac{m_{2,A}(T)}{m_{1,A}(T)}  \quad {\rm for \ } T\ge T_3 \ .
\ee
Since we have already established that  $\lim_{T\ra\infty} m_{1,A}(T)=\infty$, we conclude from (\ref{T2}) that
$m_\infty=\infty$, contradicting our original assumption. 

It follows from (\ref{O2}), and the fact that $m_{1,A}(\cdot)\ge1$, $\lim_{T\ra\infty} \frac{m_{2,A}(T)}{m_{1,A}(T)}=\infty$ that $a(T),b(T)$ in (\ref{K2}) satisfy $\lim_{T\ra\infty} a(T)/b(T)\ra\infty$ and $\lim_{T\ra\infty} a(T)=\infty$.
We see then  from  (\ref{J2}) and Lemma A.1 there are constants $C_3,T_3>0$ such that
\be \label{U2}
 E[X_{\ve,y,T} \ | \ X_{\ve,y,T}>0] \ \le \  C_3\frac{\sig_A^2(T)}{m_{2,A}(T)} \quad {\rm for \ } T\ge T_3, \ y\in[0,y_\infty] \ .
\ee
Now (\ref{I2}), (\ref{U2}) imply that $\La_\ve(T)\le  C_3\sig_A^2(T)/m_{2,A}(T)$, whence we conclude that $\lim_{T\ra\infty} \sig_A^2(T)/m_{2,A}(T)=\infty$.  From (\ref{O2}) we see  that 
$ \sig_A^2(T)/m_{2,A}(T)\le m_{1,A}(T)$, whence (\ref{U2}) implies that $\La_\ve(T)\le C_3m_{1,A}(T)$.  We have then from (\ref{O2}) there is a constant $C_4$ such that
\be \label{V2}
\frac{\sig^2_A(T)}{m_{1,A}(T)^2}  \ \le \ C_4\int_0^T \frac{dt}{m_{1,A}(t)\La_\ve(t)} \ = \ C_4\left[1-\frac{1}{m_{1,A}(T)}\right] \ .
\ee
\end{proof}
We recall from \cite{cdw} some functions associated with positive random variables $X$ with finite mean $\langle X\rangle<\infty$. Let us assume that $X$ has integrable pdf proportional to a function $c_X:(0,\infty)\ra\mathbb{R}^+$. We define functions $w_X, \ h_X$ with domain $(0,\infty)$ by 
\be \label{W2}
w_X(x) \ = \ \int_x^\infty c_X(x') \ dx' \ ,  \quad h_X(x) \ = \ \int_x^\infty w_X(x') \ dx' \ .
\ee
The {\it beta} function $\beta_X$ associated with $X$ also has domain $(0,\infty)$ and is defined by
\be \label{X2}
\beta_X(x) \ = \ \frac{c_X(x)h_X(x)}{w_X(x)^2} \ .
\ee
It is easy to see from (\ref{W2}), (\ref{X2})  that $E[X-x \ |  \ X>x]=h_X(x)/w_X(x)$ and 
\be \label{Y2}
-\frac{d}{dx} E[X-x \ | \ X>x] \ = \  1-\beta_X(x) \ . 
\ee
Furthermore, if we define $v_X(x)=E[X-x \ |  \ X>x]^{-1}$ then
\be \label{Z2}
 h_X(x) \ = \ h_X(0)\exp\left[-\int_0^x v_X(x') \ dx'\right] \  . 
\ee
\begin{proposition}
Assume $c_\ve(x,0), \ x\in\R,$ has compact support, satisfies the conditions of Lemma 2.1 and $c_\ve(\cdot,t), \ t>0$, is the solution to (\ref{A1}), (\ref{C1}). Let $X_{\ve,t}$ be the nonnegative random variable with pdf proportional to $c_\ve(x,t), \ x>0$, and $\mathcal{X}$ be the exponential variable with mean $1$.  Then 
 \be \label{AA2}
 \frac{X_{\ve,t}}{\langle X_{\ve,t}\rangle} \xrightarrow{D} \mathcal{X} \ , \quad  \ {\rm as \ \ } t\ra\infty \ .
 \ee
  In addition one has
 \be \label{AB2}
 \lim_{t\ra\infty}\|\beta_{X_{\ve,t}}(\cdot)-1\|_\infty \ = \ 0  .
  \ee
\end{proposition}
\begin{proof}
We have from Lemma 3.1 that the functions $a(\cdot),\  b(\cdot)$ of (\ref{K2}) satisfy $\lim_{T\ra\infty}a(T)=\infty, \ \lim_{T\ra\infty}b(T)/a(T)=0$. Hence $z_{y,T}$ given by (\ref{J2}) satisfies $\lim_{T\ra\infty} z_{y,T}=+\infty$, and the limit is uniform for $y\in[0,y_\infty]$.  Let $\tilde{X}_{\ve,y,T}$ be the random variable $X_{\ve,y,T}$ of (\ref{H2})  conditioned on $X_{\ve,y,T}>0$. It follows  from  Lemma A.2 that
 \be \label{AC2}
 \frac{\tilde{X}_{\ve,y,T}}{\langle \tilde{X}_{\ve,y,T}\rangle} \xrightarrow{D} \mathcal{X} \ \ {\rm as \ \ } T\ra\infty \ .
 \ee
The limit in (\ref{AB2}) is  uniform for $y\in[0,y_\infty]$. We also have from Lemma A.1 that
\be \label{AD2}
 \frac{\langle \tilde{X}_{\ve,y,T}\rangle}{\langle \tilde{X}_{\ve,0,T}\rangle}   \ \ra \ 1 \ \ {\rm as \ \ } T\ra\infty \ ,
\ee
and the limit is uniform for $y\in[0,y_\infty]$. We conclude from (\ref{I2}), (\ref{AD2}) that
\be \label{AE2}
 \frac{\langle \tilde{X}_{\ve,y,T}\rangle}{\langle X_{\ve,T} \rangle}   \ \ra \ 1 \ \ {\rm as \ \ } T\ra\infty \ ,
\ee
and the limit is uniform for $y\in[0,y_\infty]$. We have from (\ref{C2}), (\ref{H2}) that
\begin{multline} \label{AF2}
P\left(\frac{X_{\ve,T}}{\langle X_{\ve,T} \rangle}> x\right) \\
 =  \  \int_0^{y_\infty} P\left(\frac{X_{\ve,y,T}}{\langle X_{\ve,T} \rangle}> x\right) c_\ve(y,0) \ dy  \Bigg /
 \int_0^{y_\infty} P(X_{\ve,y,T}>0) c_\ve(y,0) \ dy  \\
 =  \ \int_0^{y_\infty} P\left(\frac{\tilde{X}_{\ve,y,T}}{\langle X_{\ve,T} \rangle}> x\right) P(X_{\ve,y,T}>0)c_\ve(y,0) \ dy  \Bigg /\int_0^{y_\infty} P(X_{\ve,y,T}>0) c_\ve(y,0) \ dy  \ .
\end{multline}
 The convergence in distribution (\ref{AA2}) then follows from (\ref{AC2})-(\ref{AF2}).
 
The limit of the beta function in (\ref{AB2}) already follows from Lemma 2.1 of \cite{cdw} in the case when $c_\ve(\cdot,0)$ is a point distribution.  For $c_\ve(\cdot,0)$ supported in an interval $[0,y_\infty]$ we write $\beta_\ve(x,T)=A_\ve(x,T)C_\ve(x,T)/B_\ve(x,T)^2$. The function $A_\ve$ is given by the formula
\begin{multline} \label{AG2}
A_\ve(x,T) \ = \ \int_0^{y_\infty} \exp\left[\frac{b(T)xy}{\ve\sig_A(T)}\right]\tilde{c}_\ve(y,0) \ dy \ ,  \\
{\rm where \ } \tilde{c}_\ve(y,0) \ = \  \exp\left[\frac{a(T)b(T)y}{\ve}-\frac{b(T)^2y^2}{2\ve}\right]c_\ve(y,0) \ ,
\end{multline}
and $a(\cdot),b(\cdot)$ are given by (\ref{J2}), (\ref{K2}). 
The functions $B_\ve, C_\ve$ are given by the formulae
\begin{multline} \label{AH2}
B_\ve(x,T) \ = \ \int_0^{y_\infty}dy \int_0^\infty dx'  \  \exp\left[\frac{b(T)(x+x')y}{\ve\sig_A(T)}-\frac{a(T)x'}{\ve\sig_A(T)}-\frac{x'(2x+x')}{2\ve\sig^2_A(T)}\right]\tilde{c}_\ve(y,0) \ , \\
C_\ve(x,T) \ = \ \int_0^{y_\infty}dy \int_0^\infty dx'  \ x' \exp\left[\frac{b(T)(x+x')y}{\ve\sig_A(T)}-\frac{a(T)x'}{\ve\sig_A(T)}-\frac{x'(2x+x')}{2\ve\sig^2_A(T)}\right]\tilde{c}_\ve(y,0) \ .
\end{multline}
Observe now that for $\del>0$ small one has
\begin{multline} \label{AI2}
\int_0^\infty e^{-z-\del z^2/2} \ dz \ = \ 1-\del +O(\del^2) \ , \\
\int_0^\infty ze^{-z-\del z^2/2} \ dz \ = \ 1-3\del +O(\del^2) \ .
\end{multline}
From (\ref{AH2}), (\ref{AI2}) we see there exists $T_0\ge 1$ sufficiently large such that for $T\ge T_0$, 
\begin{multline} \label{AJ2}
B_\ve(x,T) \ = \ [1-\del_\ve(x,T)]\int_0^{y_\infty} \left[\frac{a(T)-b(T)y}{\ve\sig_A(T)}+\frac{x}{\ve\sig_A^2(T)}\right]^{-1} \exp\left[\frac{b(T)xy}{\ve\sig_A(T)}\right]\tilde{c}_\ve(y,0) \ dy \ , \\
C_\ve(x,T) \ = \ [1-3\del_\ve(x,T)]\int_0^{y_\infty} \left[\frac{a(T)-b(T)y}{\ve\sig_A(T)}+\frac{x}{\ve\sig_A^2(T)}\right]^{-2} \exp\left[\frac{b(T)xy}{\ve\sig_A(T)}\right]\tilde{c}_\ve(y,0) \ dy \ .
\end{multline}
with $\|\del_\ve(\cdot,T)\|_\infty\le 2\ve/a(T)^2$.  Using the fact that $\lim_{T\ra\infty}a(T)=\infty, \ \lim_{T\ra\infty}b(T)/a(T)=0,$ we conclude from (\ref{AG2}), (\ref{AJ2}) that  $
\beta_\ve(\cdot,T)$ converges uniformly to $1$ as $T\ra\infty$.  The limit (\ref{AB2}) follows. 
\end{proof}
\begin{proposition}
Assume $c_\ve(x,0), \ x\in\R,$ has compact support, satisfies the conditions of Lemma 3.1 and $c_\ve(\cdot,t), \ t>0$, is the solution to (\ref{A1}), (\ref{C1}). Then $\lim_{t\ra\infty} \frac{d\La_\ve(t)}{dt}=1$. 
\end{proposition}
\begin{proof} We see from (\ref{E2}) and Lemma 3.1 that for any $\del>0$ there exists $T_\del\ge 1$ such that
\be \label{AK2}
\frac{\ve}{2}\left| \frac{\pa c_\ve(0,t)}{\pa x}\right| \ \le \ \del c_\ve(0,t) \quad {\rm for \ } t\ge T_\del \ .
\ee
If we also choose $T_\del$ sufficiently large so that $\ve/2\La_\ve(t)\le \del$ for $t\ge T_\del$, then we conclude from (\ref{A2}), (\ref{X2}), (\ref{AK2}) that
\be \label{AL2}
[1-2\del]\beta_{X_{\ve,t}}(0) \ \le \ \frac{d\La_\ve(t)}{dt} \ \le [1+2\del]\beta_{X_{\ve,t}}(0) \quad {\rm \ } t\ge T_\del \ .
\ee
The result follows from (\ref{AB2}) of Proposition 3.1 and (\ref{AL2}). 
\end{proof}
\begin{rem}
Observe that the rate of coarsening for the $\ve=0$ CP model is given by
\be \label{AM2}
\frac{d}{dt}\langle X_{0,t}\rangle \ = \  \beta_{X_{0,t}}(0) \ .
\ee 
Hence the formula (\ref{A2}) for the whole line $\ve>0$ model can be considered as a linear combination of the $\ve=0$ formula (\ref{AM2}) and  the $\ve>0$ formula (\ref{D1})   for the  Dirichlet boundary condition case.  
\end{rem}
\begin{rem}
It follows from (\ref{AJ2}) that
\be \label{AN2}
\lim_{T\ra\infty} \frac{\La_\ve(T)a(T)}{\ve\sig_A(T)} \ = \ \lim_{T\ra\infty} \frac{\La_\ve(T)m_{2,A}(T)}{\ve\sig^2_A(T)} \ = \ 1  \ , \quad{\rm where \ } A(\cdot)\equiv \frac{1}{\La_\ve(\cdot)} \ .
\ee
Equation (\ref{AN2}) suggests that $\La_\ve(T)/T$ converges to $1$ at a logarithmic rate. To see this note from Lemma 2.1 that 
\be \label{AO2}
\lim_{T\ra\infty} \frac{\sig_A^2(T)}{m_{1,A}(T)^2} \  = C_\ve \ < \ \infty \ .
\ee 
We conclude from (\ref{O2}), (\ref{AN2}), (\ref{AO2}) and Proposition 3.2 that
\be \label{AP2}
\lim_{T\ra\infty} \frac{T}{m_{1,A}(T)}\int_0^T\frac{ds}{m_{1,A}(s)} \ = \ \ve C_\ve \ .
\ee
If we choose $m_{1,A}(T)\simeq T(\log T)^{1/2}$ for large $T$ then the limit on the LHS of (\ref{AP2}) is  $2$ and
\be \label{AQ2}
\frac{dm_{1,A}(T)}{dT} \ \simeq \ \frac{m_{1,A}(T)}{T}\left[1+\frac{1}{2\log T}\right] \ .
\ee
From (\ref{B2}) we have that
\be \label{AR2}
\frac{dm_{1,A}(T)}{dT}  \ = \ A(T)m_{1,A}(T) \ = \ \frac{1}{\La_\ve(T)}m_{1,A}(T) \ .
\ee
Comparing (\ref{AQ2}), (\ref{AR2}) we conclude that
\be \label{AS2}
\frac{\La_\ve(T)}{T} \ \simeq \ 1-\frac{1}{2\log T}  \quad {\rm as \ } T\ra\infty \ .
\ee

\end{rem}

\section{The half line problem}
We consider now the half line problem (\ref{A1}), (\ref{B1}) with zero Dirichlet boundary condition. Our approach will be to regard this problem as a perturbation from the whole line problem studied in $\S2$. Hence it is helpful to regard the half line problem as (\ref{A1}), (\ref{C1}) rather than the equivalent (\ref{A1}), (\ref{B1}). In comparing the half line problem to the full line problem, the main difficulty is in finding estimates on the ratio of the half line Dirichlet Green's function $G_{\ve,D}$ for (\ref{A1})  to the whole line Green's function (\ref{D2}). We write the ratio as
\be \label{A3}
K_{\ve,D}(x,y,t,T) \ = \ G_{\ve,D}(x,y,t,T)\big/ G_\ve(x,y,t,T) \ , \quad x>0, \ 0<t<T \ .
\ee
In \cite{cd} we proved a log concavity property for this ratio (see Theorem 1.1 of \cite{cd}). We shall use this fact to show that the conclusions of Lemma 3.1 extend to the half line problem. 
\begin{lem}
Let $c_\ve(x,t), \ x,t>0,$ be the solution to (\ref{A1}), (\ref{B1}) with zero Dirichlet boundary condition and non-negative  initial data $c_\ve(\cdot,0)$ satisfying (\ref{F2}).  Then the function $\La_\ve(\cdot)$ of (\ref{C1}) is increasing and $\lim_{T\ra\infty}\La_\ve(T)=\infty$. If in addition $c_\ve(\cdot,0)$ has compact support, then $\lim_{T\ra\infty}m_{1,A}(T)=\infty$, $\lim_{T\ra\infty} \frac{m_{2,A}(T)}{m_{1,A}(T)}=\infty$,  $\lim_{T\ra\infty} \frac{\sig^2_A(T)}{m_{2,A}(T)}=\infty$ and $\lim_{T\ra\infty} \frac{\sig^2_A(T)}{m_{1,A}(T)^2}<\infty$ , where $A(\cdot)=1/\La_\ve(\cdot)$. 
\end{lem}
\begin{proof}
For $T>0$ we define the function $u_\ve$ by
\be \label{B3}
u_\ve(y,t,T) \ = \ \int_0^\infty G_{\ve,D}(x,y,t,T) \ dx \ , \quad  y>0, \ t<T \ .
\ee
In place of (\ref{I2})  we have  the formulas
\be \label{C3}
\int_0^\infty c_\ve(x,T) \ dx \ = \ \int_0^\infty u_\ve(y,0,T) \ c_\ve(y,0) \ dy \ ,
\ee
 and
 \be \label{D3}
\int_0^\infty xc_\ve(x,T) \ dx \ = \ \int_0^\infty E[X_{\ve,y,T} ]u_\ve(y,0,T)c_\ve(y,0) \ dy \ , 
\ee
where  $X_{\ve,y,T}$ is the positive random variable with pdf proportional to the function $x\ra G_{\ve,D}(x,y,0,T)$. 
Since the function $x\ra G_\ve(x,y,0,T)$ of (\ref{D2}) is the pdf of the random variable (\ref{H2}) we have that
\begin{multline} \label{E3}
E[X_{\ve,y,T} ] \ = \\
 \sqrt{\ve}\sig_A(T)\frac{E[K_{\ve,D}( \sqrt{\ve}\sig_A(T)\{Z-z_{y,T}/\sqrt{\ve}\},y,0,T)\{Z-z_{y,T}/\sqrt{\ve}\} \ | \ Z>z_{y,T}/\sqrt{\ve}]}
 {E[K_{\ve,D}( \sqrt{\ve}\sig_A(T)\{Z-z_{y,T}/\sqrt{\ve}\},y,0,T) \ | \ Z>z_{y,T}/\sqrt{\ve}]} \ ,
\end{multline}
where $Z$ is the standard normal variable and $K_{\ve,D}$ is given by (\ref{A3}). Observe that the function $f(z)= K_{\ve,D}( \sqrt{\ve}\sig_A(T)z,y,0,T), \ z\ge0,$ is continuous  increasing with $f(0)=0$ and $\lim_{z\ra\infty}f(z)=1$ (see Proposition 3.2 of \cite{cd}). 

It follows from (\ref{D1}) that $\La_\ve(\cdot)$ is increasing.  As in Lemma 3.1 we assume for contradiction that $\lim_{T\ra\infty}\La_\ve(T)=\La_\infty<\infty$.  Then (\ref{L2}) holds, and in place of (\ref{M2}) we have from 
(\ref{Q8}) of Lemma A.3 that $E[X_{\ve,y,T}]\ge c\sqrt{\ve}\sig_A(T), \ y>0,T\ge 1,$ for some constant $c>0$. The inequality (\ref{N2}) follows now from (\ref{C3}), (\ref{D3}), whence we conclude $\lim_{T\ra\infty}\La_\ve(T)=\infty$. To prove that $\lim_{T\ra\infty}m_{1,A}(T)=\infty$ we need to use the concavity of the function $z\ra-\log[1-f(z)]$ (see (1.17) of \cite{cd}).  In particular, we need to show that if the functions $a(\cdot),b(\cdot)$ satisfy (\ref{P2}) then the analogue  of (\ref{Q2}) given by
\be \label{F3}
 E[X_{\ve,y,T} ] \ \le \  C_2 \quad {\rm for \ } T\ge T_2, \ y\in(0,y_\infty]  \ , 
\ee
holds. Using the representation (\ref{E3}), we see that (\ref{F3}) follows from (\ref{A8}) of Lemma A.1 and (\ref{R8}) of Lemma A.3. 

The argument that $\lim_{T\ra\infty}m_{2,A}(T)/m_{1,A}(T)=\infty$ follows as in Lemma 2.1 from (\ref{N2}).  To prove that $\lim_{T\ra\infty} \sig^2_A(T)/m_{2,A}(T)=\infty$ we again need to use the concavity of the function $z\ra-\log[1-f(z)]$. Thus (\ref{A8}) and (\ref{R8}) imply the analogue of (\ref{U2}),  
\be \label{G3}
 E[X_{\ve,y,T}] \ \le \  C_3\frac{\sig_A^2(T)}{m_{2,A}(T)} \quad {\rm for \ } T\ge T_3, \ y\in(0,y_\infty] \ ,
\ee
whence we conclude from (\ref{C3}), (\ref{D3})  that $\La_\ve(T)\le C_3\sig_A^2(T)/m_{2,A}(T)$ for $T\ge T_3$.
The remainder of the argument proceeds as in Lemma 3.1. 
\end{proof}
\begin{rem}
Lemma 7.1 of \cite{cdw} gives a proof that $\lim_{T\ra\infty}\La_\ve(T)=\infty$ by using some simple inequalities for solutions to (\ref{A1}), (\ref{B1}) with zero Dirichlet condition. Note that the proof in Lemma 4.1 of $\lim_{T\ra\infty}\La_\ve(T)=\infty$ uses only the elementary property (proved using the maximum principle)  that the function $x\ra K_{\ve,D}(x,y,0,T)$ is increasing.  The proof in Lemma 4.1 of $\lim_{T\ra\infty}m_{1,A}(T)=\infty$ uses the more subtle log concavity property of the function $x\ra K_{\ve,D}(x,y,0,T)$. Evidently the result that $\lim_{T\ra\infty}m_{1,A}(T)=\infty$ is a type of upper bound on the rate of coarsening i.e. the rate at which $\La_\ve(\cdot)$ can increase.  Theorem 1.2 of \cite{cdw} yields a bound $\La_\ve(T)\le CT+\La_\ve(0)$, which implies that  $\lim_{T\ra\infty}m_{1,A}(T)=\infty$. The proof of Theorem 1.2 of \cite{cdw} also proceeds by establishing a log concavity condition on solutions to (\ref{A1}) with zero Dirichlet boundary condition. 
\end{rem}
In  Theorem 1.2 of \cite{cdw} we  obtained an upper bound on the rate of coarsening provided the initial data for (\ref{A1}), (\ref{B1}) satisfies a log concavity condition.
The condition is that the function $h_{X_{\ve,0}}(\cdot)$ defined by (\ref{W2}) for the initial condition random variable $X_{\ve,0}$ is log concave.  It was shown  in $\S7$ of \cite{cdw} that this implies  $h_{X_{\ve,t}}(\cdot)$ is also log concave for $t>0$.  Equivalently, one has that the beta function of the random variables $X_{\ve,t}, \ t>0,$ defined by (\ref{X2}) satisfies the inequality $\beta_{X_{\ve,t}}(\cdot)\le 1$. Next we obtain a bound on the beta function of the variables $X_{\ve,t}, \ t>0,$ when the initial data has compact support.
\begin{lem}
Let $c_\ve(x,t), \ x,t>0,$ be the solution to (\ref{A1}), (\ref{B1}) with zero Dirichlet boundary condition and non-negative  initial data $c_\ve(\cdot,0)$ satisfying (\ref{F2}).
Assume  $c_\ve(\cdot,0)$ has compact support and for $t>0$ denote by $X_{\ve,t}$  the random variable with pdf proportional to $c_\ve(\cdot,t)$. Then for any $T_0>0$ there is a constant $C$ such that $\beta_{X_{\ve,t}}(\cdot)\le C$ for $t\ge T_0$. 
\end{lem}
\begin{proof}
We bound $\beta_{X_{\ve,t}}(\cdot)$ from above by a constant times the beta function of the corresponding random variable  in the whole line problem. The result then follows from the argument in the proof of Proposition 3.1. To obtain the bound we note that
\be \label{H3}
c_\ve(x,T) \ = \ \int_0^{y_\infty}K_{\ve,D}(x,y,0,T)G_\ve(x,y,0,T)c_\ve(y,0) \ dy  \ ,
\ee
and write
\begin{multline} \label{I3}
w_\ve(x,T) \ = \ \int_x^\infty c_\ve(x',T) \ dx' \ = \ \int_x^\infty dx' \int_0^{y_\infty}K_{\ve,D}(x',y,0,T)G_\ve(x',y,0,T)c_\ve(y,0) \ dy  \\
\ge \  \int_0^{y_\infty}K_{\ve,D}(x,y,0,T)\left[\int_x^\infty dx' \ G_\ve(x',y,0,T)\right] c_\ve(y,0) \ dy  \ ,
\end{multline}
since the function $x\ra K_{\ve,D}(x,y,0,T)$ is increasing.  Observe that
\be \label{J3}
\frac{1}{G_\ve(x,y,0,T)}\int_x^\infty dx' \ G_\ve(x',y,0,T) \ = \ \frac{B_\ve(x,y,T)}{A_\ve(x,y,T)} \ , 
\ee
where $A_\ve(x,y,T), \ B_\ve(x,y,T)$ are the functions $A_\ve(x,T), \ B_\ve(x,T)$ of  (\ref{AG2}), (\ref{AH2}) with $\tilde{c}_\ve(\cdot,0)$ given by the Dirac delta function
concentrated at $y$.  Since Lemma 4.1 implies that $\lim_{T\ra\infty}a(T)=\infty, \ \lim_{T\ra\infty} b(T)/a(T)=0$,  we  have from (\ref{AG2}), (\ref{AJ2}) there  exists $T_1\ge T_0$ such that
\be \label{K3}
\frac{1}{G_\ve(x,y,0,T)}\int_x^\infty dx' \ G_\ve(x',y,0,T) \ \ge \ \frac{1}{2}  \left[\frac{a(T)}{\ve\sig_A(T)}+\frac{x}{\ve\sig_A^2(T)}\right]^{-1} \  \ x\ge 0, \ T\ge T_1, \ 0<y<y_\infty \ .
\ee
Arguing as in the proof of Lemma 4.1 we also see that the variable $X_{\ve,y,T}$ defined there satisfies the inequality 
\be \label{L3}
E[X_{\ve,y,T}-x \ | \ X_{\ve,y,T}>x] \ \le \  C_1\left[\frac{a(T)}{\ve\sig_A(T)}+\frac{x}{\ve\sig_A^2(T)}\right]^{-1} \  \ x\ge 0, \ T\ge T_1, \ 0<y<y_\infty \ ,
\ee
for some constant $C_1>0$.  It follows from (\ref{K3}), (\ref{L3}) that $\beta_{X_{\ve,T}}(\cdot)\le 2C_1$ if $T\ge T_1$.  The upper bound on $\beta_{X_{\ve,T}}(\cdot)$ in the region $T_0<T<T_1$ is straightforward. 
\end{proof}
\begin{proposition}
Assume  $c_\ve(\cdot,0)$ satisfies the conditions of Lemma 3.2.  Then for any $T_0>0$ there is a constant $C_0$ such that the function $\La_\ve(\cdot)$ defined by (\ref{C1})  satisfies the inequality $d\La_\ve(t)/dt\le C_0$ when $t\ge T_0$. 
\end{proposition}
\begin{proof}
We follow the arguments of Lemma 7.2 and Lemma 7.3 of \cite{cdw}, using Lemma 4.2 to substitute for the argument of Lemma 7.2. Thus by rescaling
and Lemma 4.2 we can assume  that
\be \label{M3}
\langle X_{\ve,0}\rangle \ = \ 1, \quad \beta_{X_{\ve,0}}(\cdot) \ \le \ C_1 ,
\ee
where $C_1$ is the bound obtained in Lemma 4.2. We wish to show there is a second constant $C_2$, depending only on $C_1$,  such that
\be \label{N3}
\frac{\ve}{2}\frac{\pa c_\ve(0,1)}{\pa x} \ \le \ C_2c_\ve(\ve,1) \quad {\rm if \ } \ \ve\le 1.
\ee
Expressing the function $K_{\ve,D}$ of (\ref{A3}) as
\be \label{O3}
K_{\ve,D}(x,y,0,T) \ = \ \left\{1-\exp\left[-\frac{q_\ve(x,y,T)}{\ve}\right]\right\} \ ,
\ee
we have  that
\be \label{P3}
\frac{\ve}{2}\frac{\pa c_\ve(0,1)}{\pa x} \ = \ \frac{1}{2}\int_0^{\infty} \frac{\pa q_\ve(0,y,1)}{\pa x} G_\ve(0,y,0,1)c_\ve(y,0) \ dy \ .
\ee
We have then from (\ref{P3}) and Proposition 5.1 of \cite{cd} the inequality
\be \label{Q3}
\frac{\ve}{2}\frac{\pa c_\ve(0,1)}{\pa x} \ \le \ \int_0^{\infty} \left[1+\frac{m_{1,A}(1)y}{\sig^2_A(1)}\right] G_\ve(0,y,0,1)c_\ve(y,0) \ dy \ ,
\ee
where $A(\cdot)\equiv 1/\La_\ve(\cdot)$.  Using the lower bound in Proposition 3.3 of \cite{cd} we have that
\be \label{R3}
c_\ve(\ve,1) \ \ge \ \frac{2m_{1,A}(1)}{\sig^2_A(1)} \int_0^{\infty} y\exp\left[-\frac{2m_{1,A}(1)y}{\sig^2_A(1)}\right] G_\ve(\ve,y,0,1)c_\ve(y,0) \ dy \ .
\ee
Observe that the function $A(\cdot)$ is decreasing and satisfies $A(0)=1$. The inequality (\ref{N3}) follows from (\ref{Q3}), (\ref{R3}) by using the argument in the proof of Lemma 7.2 of \cite{cdw}. The key point is that the inequality (7.6) of \cite{cdw} continues to hold.  That is if $X_{\ve,0}$ satisfies (\ref{M3}) then for any $\del$ with $0<\del<1$ there exists  a constant $\nu(\del)>0$, depending only on  $C_1$, such that
\be \label{S3}
P(X_{\ve,0}<\nu(\del)) \ \le \ \del \ .
\ee
The inequality (\ref{S3}) follows from (\ref{Y2}), (\ref{Z2}). 
Once (\ref{N3}) is established the proof of the lemma follows as in the proof of Lemma 7.3 of \cite{cdw}. 
\end{proof}
 Proposition 4.1 gives an upper bound on the rate of coarsening. We can also obtain a lower bound by using the log concavity property of the function $K_{\ve,D}$, which was crucial to the proof of Lemma 4.1.  We illustrate this first for the classical ($\ve=0$) CP model.
 \begin{lem}
 Assume $0<\ve\le 1$ and $c_\ve(x,0)=f(x)e^{-x}, \ x>0,$ is the initial data for the diffusive CP model (\ref{A1}), (\ref{B1}), where $f:[0,\infty)\ra\infty$ is  continuous non-negative increasing, satisfying $\lim_{x\ra \infty} f(x)=1$ and  with the property that the function $-\log[1-f(\cdot)]$ is concave.  Then there is a universal constant $\del_0>0$ such that $\La_\ve(1)\ge [1+\del_0]\La_\ve(0)$.
  \end{lem}
 \begin{proof}
 We first prove the result in the case $\ve=0$, so for the classical model of \cite{cp}. Let $X_{0,t}, \ t\ge 0,$ be the random variables with pdf proportional to $c_0(\cdot,t), \ t\ge 0$, where $c_0(x,t), \ x,t\ge 0,$ is the solution of the CP model (\ref{A1}), (\ref{B1}) with $\ve=0$. Then one has that
 \be \label{T3}
 \frac{d\La_0(t)}{dt} \ = \ \beta_{X_{0,t}}(0) \ ,  \quad t>0 \ .
 \ee
 Furthermore $\beta_{X_{0,t}}(x)=\beta_{X_{0,0}}(F_A(x,t)), \ x\ge 0,$ where $A(\cdot)\equiv 1/\La_0(\cdot)$ and $F_A$ is defined by
\be \label{U3}
F_A(x,t) \ = \ \frac{x+m_{2,A}(t)}{m_{1,A}(t)} \ , \quad x,t\ge0 \ .
\ee
We have from (\ref{X2}) that
\be \label{V3}
\beta_X(x) \ = \ \frac{c_X(x)}{w_X(x)} E[X-x \ | \ X>x] \ .
\ee
Arguing as in the proof of Lemma A.3, we see there is a universal constant $c_1>0$ such that $E[X_{0,0}-x \ | \ X_{0,0}>x]\ge c_1, \ x>0$.  We consider for $X=X_{0,0}$, 
\be \label{W3}
\frac{w_X(x)}{c_X(x)} \ = \ \int_x^\infty \frac{f(x')}{f(x)}e^{x-x'}\ dx' \ .
\ee
We write $f(x)=1-e^{-q(x)}$ where $q(\cdot)$ is positive increasing and concave.  If $q(x)\ge 1$ then the RHS of (\ref{W3}) is bounded above by $[1-e^{-1}]^{-1}$. If $q(x)\le 1$ then we have from (\ref{W3}) that
\be \label{X3}
\frac{w_X(x)}{c_X(x)} \ \le \ e \int_x^\infty \frac{q(x')}{q(x)}e^{x-x'}\ dx' \ .
\ee
From the concavity of $q(\cdot)$ and the fact that $q(0)\ge 0$ we obtain the inequality 
\be \label{Y3}
q(x') \ \le \ \ q(x)+\frac{q(x)}{x}(x'-x) \quad {\rm for \ } x'>x \ ,
\ee
whence the RHS of (\ref{X3}) is bounded above by $e[1+x]/x$. We conclude from (\ref{V3})-(\ref{Y3}) that
\be \label{Z3}
\beta_{X_{0,0}}(x) \ \ge \ \frac{c_1x}{e[1+x]} \ , \quad x>0 \ .
\ee
   Since the function $A(\cdot)=1/\La_0(\cdot)$ is decreasing we see that $m_{1,A}(s)\le \exp[A(0)s] , s>0,$ whence we have that
 \be \label{AA3}
\frac{m_{2,A}(t)}{m_{1,A}(t)} \ \ge \ \frac{1}{A(0)}\left[1-e^{-A(0)t}\right] \ , \quad t>0 \ .
\ee
By the argument of Lemma A.3 there are universal constants $C_2,c_2>0$ such that $c_2\le \La_0(0)\le C_2$.  The result follows from this and (\ref{Z3}), (\ref{AA3})  upon using the identity,
\be \label{AB3}
\La_0(1)-\La_0(0) \ = \ \int_0^1 \beta_{X_{0,0}}(F_A(0,t)) \ dt \ .
\ee

To extend the argument to $\ve>0$ we use (\ref{P3}) and Proposition 5.1 of \cite{cd}, whence we obtain   the inequality
\be \label{AM3}
\frac{\ve}{2}\frac{\pa c_\ve(0,t)}{\pa x} \ \ge \ \int_0^{\infty} \frac{m_{1,A}(t)y}{\sig^2_A(t)}G_\ve(0,y,0,t)c_\ve(y,0) \ dy \ .
\ee
Then (\ref{B1}), (\ref{D1}), (\ref{X2}), (\ref{Z3}), (\ref{AM3}) yield the inequality
\be \label{AN3}
\frac{d\La_\ve(t)}{dt} \ \ge \  \frac{c_1m_{1,A}(t)}{e\sig^2_A(t)}\int_0^{\infty}\frac{y^2}{1+y} G_\ve(0,y,0,t)\frac{w_X(y)^2}{h_X(0)h_X(y)} \ dy \ ,
\ee
where $X=X_{0,0}$. As in the proof of Lemma 7.3 of \cite{cdw}, we use an inequality
\be \label{AO3}
h_X(y) \ \ge \ \frac{1}{12}h_X(0) \ , \quad {\rm if \ } 0<y\le \La_0(0)/2 \ .
\ee 
We combine (\ref{AO3})  with the inequality $h_X(y)/w_X(y)=E[X-y \ | \ X>y]\le C_1$ for some constant $C_1$, as in the proof of Lemma A.3. Then we integrate (\ref{AN3}) as before to obtain the result for $0<\ve\le 1$. 
 \end{proof}
\begin{proposition}
Assume  $c_\ve(\cdot,0)$ satisfies the conditions of Lemma 3.2.  Then for any $T_0>0$ there is a constant $\del_0>0$ such that the function $\La_\ve(\cdot)$ defined by (\ref{C1})  satisfies the inequality $\La_\ve(T+\La_\ve(T))\ge [1+\del_0]\La_\ve(T)$ when $T\ge T_0$. 
\end{proposition}
\begin{proof}
We may assume that $T_0$ is sufficiently large so that we are in the asymptotic regime established in Lemma 4.1. Noting that Lemma A.3 also yields a lower bound comparable to the upper bound (\ref{L3}), we conclude that
\be \label{AC3}
E[X_{\ve,T}-x \ | \ X_{\ve,T}>x] \ \ge \  c_1\left[\frac{a(T)}{\ve\sig_A(T)}+\frac{x}{\ve\sig_A^2(T)}\right]^{-1} \  \ x\ge 0, \ T\ge T_0 \  ,
\ee
for some constant $c_1>0$. 
We wish to obtain a lower bound on $\beta_{X_{\ve,T}}(\cdot)$ similar to the one established in the proof of Lemma 4.3. This will follow from (\ref{AC3})  and an upper bound on the ratio $w_\ve(x,T)/c_\ve(x,T)$, where $c_\ve(x,T), \ w_\ve(x,T)$ are given by (\ref{H3}), (\ref{I3}). 

In analogy to (\ref{AG2}), (\ref{AH2}) we define functions $A_{\ve,D}, \ B_{\ve,D}$ by
\be \label{AD3}
A_{\ve,D}(x,T) \ = \ \int_0^{y_\infty}K_{\ve,D}(x,y,0,T) \exp\left[\frac{b(T)xy}{\ve\sig_A(T)}\right]\tilde{c}_\ve(y,0) \ dy \ , 
\ee
\begin{multline} \label{AE3}
B_{\ve,D}(x,T) \ = \ \int_0^{y_\infty}dy \int_0^\infty dx'  \ K_{\ve,D}(x+x',y,0,T) \\
\times \  \exp\left[\frac{b(T)(x+x')y}{\ve\sig_A(T)}-\frac{a(T)x'}{\ve\sig_A(T)}-\frac{x'(2x+x')}{2\ve\sig^2_A(T)}\right]\tilde{c}_\ve(y,0) \ .
\end{multline} 
Then one has for $X=X_{\ve,T}$ the formula
\be \label{AF3}
\frac{w_X(x)}{c_X(x)} \ = \ \frac{B_{\ve,D}(x,T)}{A_{\ve,D}(x,T)} \ .
\ee
Using the representation (\ref{O3}) we define $y_\infty(x,T)$ by
\be \label{AG3}
y_\infty(x,T)\ = \ \sup\{y: \ 0<y<y_\infty, \  q_\ve(x,y,T)  <\ve \ \}   \ .
\ee
Since $q_\ve(x,0,T)=0$ and the function $y\ra q_\ve(x,y,T)$ is increasing, it follows from (\ref{AG3}) that $0<y_\infty(x,T)\le y_\infty$ and $q_\ve(x,y,T)\le \ve$ if $y<y_\infty(x,T)$, with   $q_\ve(x,y,T)\ge \ve$ if $y_\infty(x,T)<y_\infty$ and $y>y_\infty(x,T)$. It follows then from (\ref{AD3}) that
\begin{multline} \label{AH3}
 A_{\ve,D}(x,T) \ \ge \ \frac{1}{\ve e} \int_0^{y_\infty(x,T)}q_\ve(x,y,T) \exp\left[\frac{b(T)xy}{\ve\sig_A(T)}\right]\tilde{c}_\ve(y,0) \ dy \\
+\left[1-e^{-1}\right] \int^{y_\infty}_{y_\infty(x,T)} \exp\left[\frac{b(T)xy}{\ve\sig_A(T)}\right]\tilde{c}_\ve(y,0) \ dy \ .
\end{multline}
Similarly we have that
\begin{multline} \label{AI3}
B_{\ve,D}(x,T) \ \le \ \frac{1}{\ve}\int_0^{y_\infty(x,T)}dy \int_0^\infty dx'  \ q_\ve(x+x',y,T) \\
\times \  \exp\left[\frac{b(T)(x+x')y}{\ve\sig_A(T)}-\frac{a(T)x'}{\ve\sig_A(T)}-\frac{x'(2x+x')}{2\ve\sig^2_A(T)}\right]\tilde{c}_\ve(y,0) \\
+\int^{y_\infty}_{y_\infty(x,T)}dy \int_0^\infty dx'  \  
\exp\left[\frac{b(T)(x+x')y}{\ve\sig_A(T)}-\frac{a(T)x'}{\ve\sig_A(T)}-\frac{x'(2x+x')}{2\ve\sig^2_A(T)}\right]\tilde{c}_\ve(y,0) \ .
\end{multline}
Applying (\ref{Y3}) to  the functions $x\ra q_\ve(x,y,T)$ and using (\ref{AI2})  we have from (\ref{AH3}), (\ref{AI3}) the inequality
\be \label{AJ3}
\frac{B_{\ve,D}(x,T)}{A_{\ve,D}(x,T)} \ \le \ C_1\left[\frac{a(T)}{\ve\sig_A(T)}+\frac{x}{\ve\sig_A^2(T)}\right]^{-1}\left\{1+
 \frac{1}{x}\left[\frac{a(T)}{\ve\sig_A(T)}+\frac{x}{\ve\sig_A^2(T)}\right]^{-1} \ \right\} \ ,
\ee
for some constant $C_1$. 

It follows from (\ref{AN2}) and Lemma A.3 that $T_0$ can be chosen sufficiently large so that
\be \label{AK3}
c_2\frac{\ve\sig_A(T)}{a(T)} \ \le \ \La_\ve(T) \ \le \ C_2\frac{\ve\sig_A(T)}{a(T)} \ , \quad {\rm for \ } T\ge T_0 \ ,
\ee
for some constants $C_2,c_2>0$. We conclude from (\ref{AC3})-(\ref{AK3}) there is a constant $c_3>0$ such that
\be \label{AL3}
\beta_{X_{\ve,T}}(x) \ \ge \ c_3\frac{x} {\La_\ve(T)+x} \ ,  \quad {\rm for \ } x>0, \ T\ge T_0 \ .
\ee
The result follows now from (\ref{AL3}) and the argument of Lemma 4.3 by scaling $\La_\ve(T)$ to $1$, whence the diffusion coefficient in (\ref{A1}) becomes $\ve/\La_\ve(T)<<1$. 
\end{proof}

\vspace{.1in}

\section{Convergence of the function $x\ra q_\ve(x,y,T)$ as $T\ra\infty$}
In Proposition 3.1 we proved convergence in distribution to the exponential variable for solutions to the whole line diffusive CP problem. This follows from the corresponding convergence in distribution of the positive random variable $\tilde{X}_{\ve,y,T}$, with density proportional to the function $x\ra G_\ve(x,y,0,T), \ x>0$.
To prove that we used certain properties of functions associated with $A(\cdot)\equiv1/\La_\ve(\cdot)$, established in Lemma 3.1. These are as follows:
\begin{multline} \label{A4}
(a)  \ \  \lim_{T\ra\infty}A(T)=0,
\quad (b) \ \  \lim_{T\ra\infty}m_{1,A}(T)=\infty, \ \quad (c) \ \ \lim_{T\ra\infty}\frac{m_{2,A}(T)}{m_{1,A}(T)}=\infty \  , \\
(d) \  \ \lim_{T\ra\infty}\frac{\sig^2_A(T)}{m_{2,A}(T)}=\infty \ , \quad  (e) \ \  \lim_{T\ra\infty}\frac{\sig^2_A(T)}{m_{1,A}(T)^2}<\infty  \ .
\end{multline}
We wish to follow a similar strategy for the half line problem. In Lemma 4.1 we showed that (\ref{A4}) with $A(\cdot)\equiv1/\La_\ve(\cdot)$  holds for the half line problem. The random variable 
$X_{\ve,y,T}$ defined in the proof of Lemma 4.1 has pdf proportional to $x\ra G_{\ve,D}(x,y,0,T), \ x>0$. The variables $\tilde{X}_{\ve,y,T}/
\langle \tilde{X}_{\ve,y,T}\rangle$ and $X_{\ve,y,T}/\langle X_{\ve,y,T}\rangle$ have 
therefore the same distributional limit as $T\ra\infty$ if we can show the ratio of their pdfs, given by the function $x\ra K_{\ve,D}(x,y,0,T)$ of (\ref{A3}),  converges to $1$ as $T\ra\infty $  for  $x$ larger than any small constant times $\min\left\{ \langle \tilde{X}_{\ve,y,T}\rangle, \     \langle X_{\ve,y,T}\rangle             \right\}$. 
In view of (\ref{O3}),  and the fact that the function $x\ra q_\ve(x,y,T)$ is increasing, this is equivalent to obtaining lower bounds on the function $q_\ve(\cdot,y,T)$.

 It was shown in $\S4$ of \cite{cd} that the function $q_0(x,y,T)=\lim_{\ve\ra0}q_\ve(x,y,T)$ is the solution to the variational problem
\begin{multline} \label{B4}
q_0(x,y,T) \ = \\
 \min\left\{\frac{1}{2}\int_\tau^T\left[\frac{dx(s)}{ds}-\la(x(s),y,s)\right]^2 \ ds \ \Big| \ 0<\tau<T, \ x(T)=x, \ x(\cdot)>0, \ x(\tau)=0 \right\} \ ,
\end{multline}
where $\la(\cdot,\cdot,\cdot)$ is defined by
\be \label{C4}
\la(x,y,s) \ = \ \left[A(s)+\frac{1}{\sig_A^2(s)}\right]x-1+\frac{m_{2,A}(s)}{\sig_A^2(s)}-\frac{m_{1,A}(s)y}{\sig_A^2(s)} \ , \quad x,y,s>0.
\ee
We see from (\ref{C4}) that if (\ref{A4}) holds then $\lim_{s\ra\infty}\la(x,y,s)=-1$.  The solution to (\ref{B4}) when $\la(\cdot,\cdot,\cdot)\equiv-1$ is given by $q_0(x,y,T)=2x$ and the optimal exit time is $\tau=\tau(x,y,T)=T-x$.  We therefore expect that $\lim_{T\ra\infty} q_0(x,y,T)=2x$ for all $x,y>0$.
The situation is however more subtle than we just described because it is possible that the minimizing trajectory $x_{\rm min}(s), \ \tau<s<T,$  for (\ref{B4}) could have $T-\tau$ large. In that case $x_{\rm min}(\cdot)$  is approximately the integral curve for the vector field $\la(\cdot,\cdot,\cdot)$ with terminal condition $x_{\rm min}(T)=x$ for a long time, but then close to time $\tau$ the trajectory $x_{\rm min}(\cdot)$ exits the positive half line with small cost.    We shall show that the conditions (\ref{A4}) rule out this possibility. 
\begin{proposition}
Assume  the function $A:[0,\infty)\ra\mathbb{R}$ is continuous positive decreasing and (\ref{A4}) holds. Then for all $x,y>0$ one has $\lim_{T\ra\infty} \frac{q_0(x,y,T)}{2x}=1$. 
In addition the limit is uniform in any region $0<x<M,  \ y_0<y<y_\infty$, where $M>0$ and $0<y_0<y_\infty<\infty$. 
\end{proposition}
 \begin{proof}
  It follows from the upper bound  $q_\ve(x,y,T)\le -2\la(0,y,T)x$, proved in Proposition 3.3 of \cite{cd}, and (\ref{A4})  that
\be \label{AB4}
\limsup_{T\ra\infty} \left[q_\ve(x,y,T)-2x\right]/x \ \le \ 0 \ ,  \quad {\rm uniformly \ for \ } 0<x<M, \ 0<y<y_\infty \ .
\ee
From Theorem 2.1 of \cite{cd} we have that $q_0(x,y,T)=\lim_{\ve\ra0}q_\ve(x,y,T)$, whence (\ref{AB4}) also holds when $\ve=0$. To obtain a lower bound
on $q_0(x,y,T)$ we consider solutions to the equation
\be \label{D4}
\frac{dx(s)}{ds} \ = \ \la(x(s),y,s) \ , \quad s<T, \ x(T)=x \ ,
\ee 
 which we denote by $x_{\rm class}(s,T), \ s<T$.  We have from (2.12) of \cite{cd} the explicit formula
 \begin{multline} \label{E4}
\sig_A^2(T)x_{\rm class}(s,T) \ = \  xm_{1,A}(s,T)\sig_A^2(s) +ym_{1,A}(s)\sig_A^2(s,T) \\
+ \ m_{1,A}(s,T)m_{2,A}(s,T)\sig^2_A(s)-m_{2,A}(s)\sig_A^2(s,T) \ ,
\end{multline}
 where the functions $(s,T)\ra m_{1,A}(s,T), \ m_{2,A}(s,T), \ \sig^2_A(s,T), \ s<T,$ are defined for the interval $[s,T]$ similarly to the corresponding functions
 $T\ra m_{1,A}(T)=m_{1,A}(0,T), \ m_{2,A}(T)=m_{2,A}(0,T), \ \sig^2_A(T)=\sig^2_A(0,T)$ of (\ref{B2}) defined for the interval $[0,T]$.  Evidently one has 
 $\lim_{s\ra T}x_{\rm class}(s,T)=x$ and $\lim_{s\ra 0}x_{\rm class}(s,T)=y$.  We can estimate $x_{\rm class}(s,T)$ in the interval $0<s<T$ using the properties (\ref{A4}) by observing from (\ref{E4}) that
 \begin{multline} \label{F4}
 x_{\rm class}(s,T) \ \ge \ \frac{\sig_A^2(s)}{m_{1,A}(s)^2}
 \left(  \frac{\sig_A^2(T)}{m_{1,A}(T)^2}   \right)^{-1}\frac{m_{2,A}(s,T)}{m_{1,A}(s,T)}  \\
 - \  \frac{1}{m_{1,A}(s)}\frac{\sig_A^2(s,T)}{m_{1,A}(s,T)^2}
 \left(  \frac{\sig_A^2(T)}{m_{1,A}(T)^2}   \right)^{-1}\frac{m_{2,A}(s)}{m_{1,A}(s)} \  , \quad 0<s<T \ .
 \end{multline}
 From (\ref{A4}) (c),(e) the first term on the RHS of (\ref{F4}) goes to $\infty$ as $T\ra\infty$, whereas from (\ref{A4}) (e) the second term converges to a finite number. 
 We conclude that $\lim_{T\ra\infty} x_{\rm class}(s,T) =\infty$. 
 
 We can use the method of characteristics to construct a solution to the Hamilton-Jacobi equation (2.29) of \cite{cd} , corresponding to the variational problem (\ref{B4}), in a neighborhood of the line $\{[x.T]\in\mathbb{R}^2: \ x=0, \ T>0\}$.  The characteristic equation, given by (4.28) of \cite{cd}, is
 \be \label{G4}
\frac{dx(s)}{ds} \ = \ \left[A(s)+\frac{1}{\sig^2_A(s)}\right]x(s) +\frac{m_{1,A}(s)}{\sig^2_A(s)}\left[y+2g_{2,A}(\tau,\tau)-g_{2,A}(s,s)\right] \ , \quad s>\tau, \ x(\tau)=0 \ ,
\ee
where the function $s\ra g_{2,A}(s,s)$ is given by the formula
\be \label{H4}
g_{2,A}(s,s) \ = \ \frac{\sig^2_A(s)-m_{2,A}(s)}{m_{1,A}(s)} \ , \quad s>0 \ .
\ee
 Differentiating (\ref{H4}) we find that
 \be \label{I4}
\frac{d}{ds} g_{2,A}(s,s) \ =  \frac{A(s)\sig^2_A(s)}{m_{1,A}(s)} \ .
\ee
Evidently the function $s\ra g_{2,A}(s,s), \ s>0,$ is non-negative increasing, and from (\ref{A4})  (c), (d) we have $\lim_{s\ra\infty} g_{2,A}(s,s)=\infty$.  Choosing $T_1>0$ so that
$\sig^2_A(s)\ge 2 m_{2,A}(s)$ for $s\ge T_1$,  we see from (\ref{H4}), (\ref{I4})  that
\be \label{J4}
\frac{d}{ds} g_{2,A}(s,s) \ \le  \ 2A(s) g_{2,A}(s,s) \ , \quad s>T_1 \ .
\ee
Integrating (\ref{J4}) we conclude that
 \be \label{K4}
 g_{2,A}(s,s) \ \le \ \exp\left[2\int_T^s A(s') \ ds'\right] g_{2,A}(T,T) \ , \quad {\rm for \ }  s>T\ge T_1 \ .
 \ee
 Since the function $A(\cdot)$ is decreasing it follows from (\ref{G4}), (\ref{K4}) that the characteristic curve $s\ra x_{\rm char}(T,s), \ s>T, \  x_{\rm char}(T,T)=0,$
 is increasing provided $T>T_1$ and $T<s<T+ \log 2/2A(T)$. We may also obtain a lower bound on $x_{\rm char}(T,s)$ by observing from (\ref{G4}), (\ref{K4}) that
 \be \label{L4}
 \frac{d}{ds} x_{\rm char}(T,s) \ \ge \ \frac{m_{1,A}(s)}{3\sig^2_A(s)}g_{2,A}(s,s) \ \ge \  \frac{1}{6} \ , \quad {\rm for \ } T<s<T+ \log (3/2)/2A(T) \ ,  \ T\ge T_1,
 \ee
 whence we conclude that
 \be \label{M4}
  x_{\rm char}(T,s) \ \ge \  \frac{s-T}{6} \ , \quad {\rm if \ } T<s<T+ \log (3/2)/2A(T) \ ,  \ \ T\ge T_1 \ .
 \ee
 
 The first variation equation for the characteristics is obtained by differentiating (\ref{G4}) with respect to $\tau$. This yields the equation
 \begin{multline} \label{N4}
\frac{d}{ds}D_T x_{\rm char}(T,s) \ = \ \left[A(s)+\frac{1}{\sig^2_A(s)}\right]D_T x_{\rm char}(T,s) +\frac{2m_{1,A}(s)}{\sig^2_A(s)}  \frac{A(T)\sig^2_A(T)}{m_{1,A}(T)} \ , \ s>T \ , \\
D_Tx_{\rm char}(T,s)\Big|_{s=T} \ = \ -\frac{m_{1,A}(T)}{\sig^2_A(T)}\left[y+g_{2,A}(T,T)\right] \ .
\end{multline}
Using the fact that the function $s\ra m_{1,A}(s)/\sig^2_A(s), \ s>0,$ is decreasing we have from (\ref{N4}) that
\begin{multline} \label{O4}
\frac{d}{ds}D_T x_{\rm char}(T,s) \ \le \ 2A(T) \quad {\rm if \ } s>T \quad {\rm and \ }  D_T x_{\rm char}(T,s) \ \le \ 0, \\
 D_T x_{\rm char}(T,s)\Big|_{s=T} \ \le \ -\frac{1}{2} \ , \quad {\rm if \ } T\ge T_1 \ .
\end{multline}
It follows from (\ref{O4})  that $D_T x_{\rm char}(T,s)<0$ if $T<s<T+1/4A(T), \ T>T_1$.  We define a domain $\mathcal{D}_y(T_1)$ by
\be \label{P4}
\mathcal{D}_y(T_1)  \ = \ \{[x,s] : \ x= x_{\rm char}(T,s), \ T<s<T+1/5A(T), \ T>T_1 \ \} \ .
\ee
Letting $\mathcal{U}(T_1)=\{[s,T]: \ T<s<T+1/5A(T), \ T>T_1 \ \}$, we have shown that the mapping $\mathcal{U}(T_1)\ra \mathcal{D}_y(T_1)$
defined by $[s,T]\ra [x_{\rm char}(T,s),s]$ is a diffeomorphism. 

One can use the method of characteristics to construct a $C^1$ solution $q_{\rm char}(x,y,T),  \ [x,T]\in \mathcal{D}_y(T_1)$, of the Hamilton-Jacobi equation
in the domain  $\mathcal{D}_y(T_1)$. From (4.26) of \cite{cd} we have the formula
\be \label{Q4}
\frac{\pa q_{\rm char}(x_{\rm char}(T,s),y,s)}{\pa x} \ = \ \frac{2m_{1,A}(s)[y+g_{2,A}(T,T)]}{\sig^2_A(s)}  \ , \quad s>T \ .
\ee
In view of (\ref{A4}) (a) and (\ref{M4}), we see that for any $M>0$ there exists $T_M\ge T_1$ such that  $\mathcal{D}_y(T_1)$ contains the infinite rectangle
$\{[x,T]: \ 0<x<M, \ T>T_M\}$.  Furthermore, (\ref{A4}) (a), (b), (d), (e) and (\ref{M4}), (\ref{Q4}) imply for any $y_\infty>0$  that
\be \label{R4}
\lim_{T\ra\infty} \frac{\pa q_{\rm char}(x,y,T)}{\pa x} \ = \ 2 \ ,  \quad {\rm uniformly \ for \ } 0<x<M, \ 0<y<y_\infty \ .
\ee
As a consequence of (\ref{R4}) we have that $\lim_{T\ra\infty} q_{\rm char}(x,y,T)/2x=1$ for all $x,y>0$.  Let $x(s), \ \tau< s\le T,$ be a path in $\mathbb{R}^+$ such that $x(T)=x$ and with first exit time $\tau>0$ from $\mathbb{R}^+$. The associated Lagrangian $\mathcal{L}$ and action integral $\mathcal{A}$  are given by the expressions
\begin{multline} \label{S4}
 \mathcal{L}(x(\cdot), x,y,s,T)  \ = \ \frac{1}{2}\left[\frac{dx(s)}{ds}-\la(x(s),y,s)\right]^2 \ ,  \quad \tau<s<T \ , \\
  \mathcal{A}(x(\cdot), x,y,T) \ = \ \int_\tau^T  \mathcal{L}(x(\cdot), x,y,s,T) \ ds  \ .
\end{multline}
The usual verification theorem (see Proposition 4.2 of \cite{cd})  implies that if the path $s\ra [x(s),s],  \ \tau<s\le T,$ lies in $\mathcal{D}_y(T_1)$ then 
$\mathcal{A}(x(\cdot), x,y,T)\ge q_{\rm char}(x,y,T)$. Suppose now that $T>T_M$ and  the path exits  $\mathcal{D}_y(T_1)$, but then reenters  $\mathcal{D}_y(T_1)$ at time $\tau^*>T_M$ and  remains in $\mathcal{D}_y(T_1)$, until it exits $\mathbb{R}^+$. In that case we have the inequality  $\mathcal{A}(x(\cdot), x,y,T)\ge q_{\rm char}(M,y,\tau^*)$.

Next we wish to obtain a lower bound on the action in the case $0<\tau<T_1$.  We observe that the path $x(s), \ \tau\le s\le T,$ is the solution to the terminal value problem
  \be \label{T4}
  \frac{dx(s)}{ds} \ = \ \la(x(s),y,s)-f(s) \ , \quad \tau<s\le T \ ,   \ x(T)=x, \ \frac{1}{2} f(s)^2 \ = \  \mathcal{L}(x(\cdot), x,y,s,T) \ .
 \ee
 It follows from (\ref{D4}) that $\tilde{x}(s)=x(s)-x_{\rm class}(s,T)$ is a solution to the terminal value problem 
 \be \label{U4}
  \frac{d\tilde{x}(s)}{ds} \ = \ \left[A(s)+\frac{1}{\sig^2_A(s)}\right]\tilde{x}(s)-f(s) \ , \quad  \tau<s<T, \ \tilde{x}(T)=0 \ .
 \ee
 Integrating (\ref{U4}) we obtain the integral formula
 \be \label{V4}
 \tilde{x}(s) \ = \  \frac{\sig_A^2(s)}{m_{1,A}(s)} \int_s^T \frac{m_{1,A}(s')}{\sig^2_A(s')}f(s') \ ds' \  .
 \ee
 Applying the Schwarz inequality in (\ref{V4}) we have from (\ref{A4}) (e) and (\ref{T4}) that
 \be \label{W4}
 \tilde{x}(s)^2 \ \le \  C\left[\int_{s}^T\frac{|f(s)|}{m_{1,A}(s,s')} \ ds' \right]^2 \ \le \ 
 2C\frac{\sig_A^2(s,T)}{m_{1,A}(s,T)^2} \int_{s}^T \mathcal{L}(x(\cdot), x,y,s',T)  \ ds' \  ,
 \ee
  for some constant $C$.  It follows from (\ref{B4}), (\ref{S4}), (\ref{W4}) that 
  \be \label{X4}
  \mathcal{A}(x(\cdot), x,y,T)  \ \ge \ \frac{m_{1,A}(s,T)^2}{2C\sig_A^2(s,T)}[x(s)-x_{\rm class}(s,T)]^2+q_0(x(s),y,s) \ , \quad \tau<s<T \ .
  \ee
  
  We wish to show that
\be \label{Y4}
\liminf_{T\ra\infty} \frac{q_0(x,y,T)}{2x} \ \ge \ 1 \ ,  \quad {\rm uniformly \ for \ } 0<x<M, \ y_0<y<y_\infty \ .
\ee
To do this we first note the lower bound from Proposition 4.1 of \cite{cd},
\be \label{Z4}
q_0(x,y,T) \ \ge \ \ \frac{2m_{1,A}(T)xy}{\sig^2_A(T)} \ , \quad x,y,T>0 \ .
\ee 
Consider a path $x(s), \ \tau<s<T,$ with $x(T)=x$ and first exit time $\tau$, where $\tau<T_{2M}<T$. Taking $s=T_{2M}$ in (\ref{X4}) and recalling that $\lim_{T\ra\infty} x_{\rm class}(T_{2M},T)=\infty$,  it follows  from (\ref{A4}) (e) and (\ref{X4}), (\ref{Z4}) that there exists $T^*_{2M}>T_{2M}$ such that
\be \label{AA4}
 \mathcal{A}(x(\cdot), x,y,T)  \ \ge  \ 2x \quad {\rm for \ } T>T^*_{2M},  \ 0<x<M, \ y>y_0 \ .
\ee
 We have already seen that the method of characteristics yields a lower bound on the action in the case $\tau>T_{2M}$. Combining that with (\ref{AA4}) implies (\ref{Y4}). We have already established the upper bound $\limsup_{T\ra\infty} q_0(x,y,T)/2x \ \le \ 1 $ in (\ref{AB4}). It also follows from the inequality $q_0(x,y,T)\le q_{\rm char}(x,y,T)$ and (\ref{R4}). 
 \end{proof}
 Next we  show that the result of Proposition  5.1 extends to the case $\ve>0$.
 We have already obtained an upper bound in (\ref{AB4}),so we just need to
obtain a lower bound on $q_\ve(x,y,T)$. We consider solutions $X_\ve(\cdot)$ to the SDE
\be \label{AC4}
dX_\ve(s) \ = \ \mu_\ve(X_\ve(s),y,s) \ ds+\sqrt{\ve}dB(s)\  , 
\ee
run {\it backwards} in time with controller $\mu_\ve$ and given terminal data. The optimal controller for the stochastic control problem corresponding to the function $[x,T]\ra q_\ve(x,y,T)$ is
\be \label{AD4}
\mu^*_\ve(x,y,T) \ = \ 
\la(x,y,T)+\frac{\pa q_\ve(x,y,T)}{\pa x} \ , \quad x,T>0 \ .
\ee
Letting $X^*_\ve(\cdot)$ be solutions to (\ref{AC4}) with $\mu_\ve=\mu^*_\ve$, we have from Lemma 2.1 and Lemma 2.3 of \cite{cd} the identity
\be \label{AE4}
q_\ve(x,y,T) \ = \ E\left[  \frac{1}{2}\int_{\tau^*_{\ve,x,T}}^T    \left[\mu^*(X^*_\ve(s),y,s)-\la(X^*_\ve(s),y,s)\right]^2 \ ds  \ \Big| \ X^*_\ve(T)=x   \right] \ ,
\ee
where $\tau^*_{\ve,x,T}$ is the first exit time of $X^*_\ve(s), \ s<T,$ with $X^*_\ve(T)=x$ from the  half line $(0,\infty)$.  Lemma 2.3 of \cite{cd} also establishes that $\tau^*_{\ve,x,T}>0$ with probability $1$. 

We generalize the lower bound (\ref{X4}) on the action. To do this we denote by $X_{\ve,{\rm class}}(s,T), \ s<T,$ the solution to (\ref{AC4}) with $\mu_\ve=\la$ and terminal condition $X_{\ve,{\rm class}}(T,T)=x$. The SDE (\ref{AC4}) is then linear and may be explicitly solved, whence we have that
\be \label{AF4}
X_{\ve,{\rm class}}(s,T) \ = \ x_{\rm class}(s,T)-\sqrt{\ve}\frac{\sig^2_A(s)}{m_{1,A}(s)}Z(s) \ , \quad {\rm with \ } Z(s)=\int_s^T\frac{m_{1,A}(s') \ dB(s')}{\sig^2_A(s')} \ , 
\ee
where $x_{\rm class}(s,T)$ is given by (\ref{E4}).  Since the function $x\ra q_\ve(x,y,T)$ is increasing it follows from (\ref{AD4}) that $\mu^*(x,y,s)\ge \la(x,y,s), \ x,y,s>0$, whence  $X_{\ve,{\rm class}}(s,T)>0$ for $\tau^*_{\ve,x,T}<s<T$. Integrating (\ref{AC4}) over an interval $[s,T]$ we have similarly to (\ref{U4}), (\ref{V4})  the identity
\begin{multline} \label{AG4}
X_{\ve,{\rm class}}(s,T)-X^*_\ve(s) \ = \\ \frac{\sig^2_A(s)}{m_{1,A}(s)}\int_s^T \frac{m_{1,A}(s')}{\sig^2_A(s')}  \left[\mu^*(X^*_\ve(s'),y,s')-\la(X^*_\ve(s'),y,s')\right] \ ds' \ , \quad \tau^*_{\ve,x,T}<s<T \ .
\end{multline}
Let $\tau_T$ be a stopping time for the diffusion (\ref{AC4}), run backwards in time with terminal time $T$, such that $\tau_T>\tau^*_{\ve,x,T}$. Applying the Schwarz inequality in (\ref{AG4}), we conclude from (\ref{AE4}) that
\begin{multline} \label{AH4}
q_\ve(x,y,T) \ \ge \ E\left[ q_\ve(X^*_\ve(\tau_T),y,\tau_T) \ \Big| \ X^*_\ve(T)=x\right] \\
+\frac{\sig^2_A(T)}{2}E\left[ \frac{\left\{X_{\ve,{\rm class}}(\tau_T,T)-X^*_\ve(\tau_T)\right\}^2}{\sig^2_A(\tau_T)\sig^2_A(\tau_T,T)} \ \Big| \ X^*_\ve(T)=x\right]  \ ,
\end{multline}
where we have used the identity
\be \label{AI4}
\int_s^T\frac{m_{1,A}(s')^2 \ ds'}{\sig^4_A(s')} \ = \ \frac{m_{1,A}(s)^2\sig^2_A(s,T)}{\sig^2_A(T)\sig^2_A(s)} \ .
\ee

In order to proceed further we first consider a simpler problem in which the function $\la$ is replaced by a constant $-k$ with $k>0$.  We see from (\ref{O3}) and
(2.18), (2.19) of \cite{cd} that the function $v_\ve$, defined by
\be \label{AJ4}
v_\ve(x,y,T) \ = \ 1-K_{\ve,D}(x,y,0,T) \ = \ \exp\left[-\frac{q_\ve(x,y,T)}{\ve}\right] \ ,
\ee
is a solution to the PDE
\be \label{AK4}
\frac{\pa v_\ve(x,y,T)}{\pa T} \ = \ -\la(x,y,T)\frac{\pa v_\ve(x,y,T)}{\pa x}+\frac{\ve}{2}\frac{\pa^2 v_\ve(x,y,T)}{\pa x^2} \ ,
\ee
with boundary condition  $\lim_{x\ra 0}v_\ve(x,y,T)=1$. 
Replacing the function $\la$ in (\ref{AK4})  by the constant  $-k$, we are interested in solutions $(x,T)\ra v_\ve(x,T)$ to the Dirichlet boundary value problem
\be \label{AL4}
\frac{\pa v_\ve(x,T)}{\pa T} \ = \ k\frac{\pa v_\ve(x,T)}{\pa x}+\frac{\ve}{2}\frac{\pa^2 v_\ve(x,T)}{\pa x^2} \ , \ \ x,T>0, \quad v_\ve(0,T)=1, \ T>0 \ .
\ee
As in (\ref{AJ4}) we define a function $(x,T)\ra q_\ve(x,T)$ by
\be \label{AM4}
v_\ve(x,T) \ = \  \exp\left[-\frac{q_\ve(x,T)}{\ve}\right] \ , \quad x,T>0 \ .
\ee
Observe that if we set $q_\ve(x,T)=2kx$ in (\ref{AM4}) then $(x,T)\ra v_\ve(x,T)$ is a solution to (\ref{AL4}). 
 \begin{proposition} 
 Let the function $(x,T)\ra v_\ve(x,T)$ be the solution to the boundary value problem (\ref{AL4})  with   initial data $v_\ve(x,0)=\exp[-q_\ve(x,0)/\ve], \ x>0$, and assume there is a non-negative function $f:\mathbb{R}^+\ra\mathbb{R}$, independent of $\ve$, with the property that $\lim_{x\ra\infty}f(x)=\infty$ and $q_\ve(x,0)\ge f(x),  \ x\ge 0$.  Then the function $q_\ve$ defined by (\ref{AM4}) has the property 
$\lim_{T\ra\infty} \left[q_\ve(x,T)-2kx\right]/x= 0$ if $x>0$. 
In addition for any $M,\ve_0>0$, the limit is uniform in the region $0<x\le M, \ 0<\ve\le \ve_0$.
\end{proposition}
 \begin{proof}
 The solution to (\ref{AL4}) with initial data $v_\ve(\cdot,0)$ has the representation
 \be \label{AN4}
 v_\ve(x,T) \ = \  \frac{\ve}{2}\int_0^T\frac{\pa G_{\ve,D}(x,0,t)}{\pa x'}  \ dt+\int_0^\infty G_{\ve,D}(x,x',T) v_\ve(x',0) \ dx' \ , \quad x,T>0 \ ,
 \ee
 where $G_{\ve,D}$ is the Dirichlet Green's function for the PDE (\ref{AL4}).  In the case of (\ref{AL4}) there is an explicit formula for $G_{\ve,D}$,
 \be \label{AO4}
G_{\ve,D}(x,x',t) \ = \ \frac{1}{\sqrt{2\pi\ve t}}\exp\left[-\frac{(x-x'+kt)^2}{2\ve t}\right]\left\{1-\exp\left[-\frac{2xx'}{\ve t}\right]\right\} \ .
\ee
We first observe that
\be \label{AP4}
 \frac{\ve}{2}\int_0^\infty\frac{\pa G_{\ve,D}(x,0,t)}{\pa x'}  \ dt \ = \ \exp\left[-\frac{2kx}{\ve}\right] \ .
\ee
In fact the LHS of (\ref{AP4}) is given by
\begin{multline} \label{AQ4}
\int_0^\infty  \frac{x}{t^{3/2}\sqrt{2\pi\ve} }\exp\left[-\frac{(x+kt)^2}{2\ve t}\right] \ dt \\
= \ \exp\left[-\frac{kx}{\ve}\right] \int_0^\infty  \frac{x}{t^{3/2}\sqrt{2\pi\ve} }\exp\left[-\frac{x^2}{2\ve t}-\frac{k^2t}{2\ve}\right] \ dt
 \ = \exp\left[-\frac{kx}{\ve}\right]2\sqrt{z}\int_0^\infty sg(s)e^{-zs} \ ds \ ,
\end{multline}
where $z, g(\cdot)$ are given by the formulae
\be \label{AR4}
z \ = \ \frac{k^2x^2}{\ve^2}  , \quad g(s) \ = \ \frac{1}{2\sqrt{\pi}s^{3/2}}\exp\left[-\frac{1}{4s}\right] \ .
\ee
Now (\ref{AP4}) follows since the Laplace transform $\mathcal{L}g(z)=\exp[-\sqrt{z}]$.

To obtain an upper bound on $\limsup_{T\ra\infty} q_\ve(x,T)/x$ we show that for any $\del>0$ there exists $T_\del>0$ such that
\be \label{AS4}
v_\ve(x,T) \ \ge \   \exp\left[-\frac{2kx(1+\del)}{\ve}\right]  \quad {\rm if \ } 0<x\le M,  \ 0<\ve\le \ve_0, \ T\ge T_\del \ .
\ee
The inequality (\ref{AS4}) follows from (\ref{AP4})  and the inequality
\be \label{AT4}
 \frac{\ve}{2}\int_T^\infty    \frac{\pa G_{\ve,D}(x,0,t)}{\pa x'}    \ dt         \ \le \ \left(\frac{2\ve}{\pi}\right)^{1/2}\frac{x}{k^2T^{3/2}} \exp\left[-\frac{k^2T}{2\ve}\right]  \ .
\ee
To obtain a lower bound on  $\liminf_{T\ra\infty} q_\ve(x,T)/x$,  we show that for any $\del>0$ there exists $T_\del>0$ such that
\begin{multline} \label{AU4}
\int_0^\infty G_{\ve,D}(x,x',T)v_\ve(x',0) \ dx'  \ \le \ e^{-2kx/\ve}\left[e^{\del x/\ve}-1\right]  \ = \ e^{-(2k-\del)x/\ve}\left[1-e^{-\del x/\ve}\right] \\
\le \ \frac{\del x}{\ve} e^{-(2k-\del)x/\ve} \quad  {\rm if \ } T\ge T_\del, \ 0<x\le M, \ 0<\ve\le \ve_0 \ .
\end{multline}
To prove (\ref{AU4}) we first observe there exists $T_0>0$ such that
\be \label{AV4}
\inf_{x'>0}\left[ \frac{(x-x'+kT)^2}{2 T}+f(x')\right] \ \ge 2kM \quad {\rm for \ all  \ } x>0,  \ \ {\rm if \ } T\ge T_0 \ .
\ee
Evidently $T_0$ must satisfy the inequalities
\be \label{AW4}
T_0 \ \ge \  \frac{4M}{k} \ , \quad f(kT_0) \ \ge \ 2kM \ .
\ee
From (\ref{AO4}), (\ref{AV4}) we see that the LHS of (\ref{AU4}) is bounded above by 
\be \label{AX4}
e^{-(2k-\del)M/\ve}  \frac{1}{\sqrt{2\pi\ve T}}\int_0^\infty dx'  \ \frac{2xx'}{\ve T}\exp\left[-\frac{\del}{2k\ve}\left\{\frac{(x-x'+kT)^2}{2 T}+f(x')\right\}\right]  \ ,
\ee
provided $T\ge T_0$. 
It follows from (\ref{AU4}), (\ref{AX4}) it is sufficient to choose $T_\del\ge T_0$ such that
\be \label{AY4}
\frac{1}{\sqrt{2\pi\ve T}}\int_0^\infty dx'  \ \frac{2x'}{T}\exp\left[-\frac{\del}{2k\ve}\left\{\frac{(x-x'+kT)^2}{2 T}+f(x')\right\}\right] \ \le \ \del  \ 
\ee
if $T\ge T_\del$.  Evidently $T_\del$ may be chosen uniformly for $0<x\le M, \ 0<\ve\le \ve_0$. 
  \end{proof}
  The proof of Proposition 5.2 depends heavily on using the explicit formula (\ref{AO4}) for the Dirichlet Green's function.   No such formula exists in the general case of the drift (\ref{C4}).  We therefore give an alternative proof of Proposition 5.2, using the representation of $q_\ve(x,T)$ as the cost function of a stochastic control problem. We will then generalize this approach to the case of the drift $\la$ of (\ref{C4}). 
  \begin{proof}[Proof of Proposition 5.2 upper bound]
  
  As in the proof of Lemma 2.1 of \cite{cd}, we have the inequality
 \begin{multline} \label{AZ4}
q_\ve(x,T) \ \le \ E\left[  \frac{1}{2}\int_{\tau_{\ve,x,T}\vee 0}^T    \left[\mu_\ve(X_\ve(s),s)+k\right]^2 \ ds  \ \Big| \ X_\ve(T)=x   \right] \\
+  E\left[  q_\ve(X_\ve(\tau_{\ve,x,T}\vee 0),\tau_{\ve,x,T}\vee 0)  \ \Big| \ X_\ve(T)=x  \right] \ .
\end{multline}
  where $X_\ve(s), \ s\le T,$ is the solution to the SDE (\ref{AC4}), and $\tau_{\ve,x,T}$ is the  first exit time of $X_\ve(\cdot)$ from the half line $(0,\infty)$.   We define the drift $\mu_\ve$ by
  \be \label{BA4}
  \mu_\ve(x,s) \ = \ k \ \ {\rm if \ } 1<s<T, \quad  \mu_\ve(x,s) \ = \ \frac{x}{s} \ \ {\rm if \ } 0<s<1 \ .
  \ee
Arguing as in the proof of Lemma 2.1 of \cite{cd} , we see that $\tau_{\ve,x,T}>0$ with probability $1$.   Hence the second term on the RHS of (\ref{AZ4}) is zero. 
It follows then from (\ref{AZ4}) that
\be \label{BB4}
q_\ve(x,T) \ \le 2k^2E[T-\tau^a_{\ve,x,T}]+E\left[  \frac{1}{2}\int_{\tau_{\ve,x,T}\wedge 1}^1    \left[\mu_\ve(X_\ve(s),s)+k\right]^2 \ ds  \ \Big| \ X_\ve(T)=x   \right] \ ,
\ee
  where $\tau^a_{\ve,x,T}$ is the first exit time from the half line $(0,\infty)$ for the diffusion with the constant drift $k$. It is easy to see that $E[T-\tau^a_{\ve,x,T}]=x/k$.  To estimate the second term on the RHS of (\ref{BB4})  we consider a diffusion run backwards in time from $s=1$ conditioned on $X_\ve(1)=x$, which satisfies the SDE (\ref{AC4}) with $\mu_\ve(x,s)=x/s$.  The solution to (\ref{AC4}) is then given by the formula
  \be \label{BC4}
  X_\ve(s) \ = \ s\left[x-\sqrt{\ve}Z(s)\right], \quad Z(s)=\int_s^1 \frac{dB(s')}{s'} \ .
  \ee
  Let $\tau_{\ve,x}<1$ be the first exit time from $(0,\infty)$ for $X_\ve(\cdot)$ with $X_\ve(1)=x$. Then using (\ref{BC4}) and arguing as in  the proof of Lemma 2.2 of \cite{cd}, we obtain   the identity
  \begin{multline} \label{BD4}
  E\left[  \frac{1}{2}\int_{\tau_{\ve,x}}^1    \left[X_\ve(s)/s+k\right]^2 \ ds  \ \Big| \ X_\ve(1)=x   \right] \ \\
  = \  \frac{(x+k)^2}{2}E[1-\tau_{\ve,x}]+x(x+k)E[\tau_{\ve,x}]
  +\frac{\ve}{2}E\left[ \int_{\tau_{\ve,x}}^1Z(s)^2 \ ds\right] \ .
  \end{multline}
  Following the argument of \cite{cd}, we see that 
  \be \label{BE4}
  E\left[ \int_{\tau_{\ve,x}}^1Z(s)^2 \ ds\right] \ \le \ C\left[ 1+|\log\ve|+x\right] \ ,
  \ee
 with constant $C$ independent of $\ve,x$ for $x>0, \ 0<\ve\le\ve_0$. Since $0<\tau_{\ve,x}<1$ one obtains from (\ref{BD4}), (\ref{BE4}) an upper bound on the LHS of (\ref{BD4}), which we denote by $F_\ve(x)$.  It follows then from (\ref{BB4}), noting that (\ref{AO4}) with $k$ replaced by $-k$ and $t=T-1$ gives the distribution  of $X_\ve(1)$ on paths $X_\ve(s)$ with $\tau_{\ve,x,T}<1$, that
  \begin{multline} \label{BF4}
  \frac{q_\ve(x,T)}{x} \ \le \ 2k+ \\
  \frac{1}{\sqrt{2\pi\ve (T-1)}}\int_0^\infty \exp\left[-\frac{\{x-x'-k(T-1)\}^2}{2\ve (T-1)}\right] \frac{1}{x}\left\{1-\exp\left[-\frac{2xx'}{\ve (T-1)}\right]\right\} F_\ve(x') \ dx'  \ .
  \end{multline} 
  The integral in (\ref{BF4}) converges to $0$ as $T\ra\infty$, uniformly for $0<\ve\le \ve_0, \ 0<x\le M$. We conclude that $\limsup_{T\ra\infty} q_\ve(x,T)/x\le 2k$ and the limit is uniform for $0<\ve\le \ve_0, \ 0<x\le M$.
  \end{proof}
  \begin{proof}[Proof of Proposition 5.2 lower bound] 
  Let $\mu_\ve^*$ be the optimal controller
  \be \label{BG4}
\mu^*_\ve(x,T) \ = \ 
-k+\frac{\pa q_\ve(x,T)}{\pa x} \ , \quad x,T>0 \ ,
\ee
and $X^*_\ve(\cdot)$ be the solution to the SDE (\ref{AC4}) with $\mu_\ve=\mu^*_\ve$. Denoting by $\tau^*_{\ve,x,T}$ the first exit time from the half line 
$(0,\infty)$ for the diffusion $X^*_\ve(s), \ s\le T,$ with $X^*_\ve(T)=x>0$, we have  an identity similar to (\ref{AE4}), 
  \begin{multline} \label{BH4}
q_\ve(x,T) \ = \ E\left[  \frac{1}{2}\int_{\tau^*_{\ve,x,T}\vee 0}^T    \left[\mu^*_\ve(X_\ve(s),s)+k\right]^2 \ ds  \ \Big| \ X^*_\ve(T)=x   \right] \\
+  E\left[  q_\ve(X_\ve(\tau^*_{\ve,x,T}\vee 0),\tau^*_{\ve,x,T}\vee 0)  \ \Big| \ X^*_\ve(T)=x  \right] \ .
\end{multline}
Let $X_{\ve,{\rm class}}(s), \ s\le T,$ be the solution to (\ref{AC4}) with $\mu_\ve\equiv-k$. Evidently, conditioned on $X_{\ve,{\rm class}}(T)=x$, one has the formula
\be \label{BI4}
X_{\ve,{\rm class}}(s) \ = \ x+k(T-s)+\sqrt{\ve}B(s), \quad s\le T \ .
\ee
Arguing as in (\ref{AG4}), (\ref{AH4}) we obtain from (\ref{BH4}) the lower bound
\be \label{BJ4}
q_\ve(x,T) \ \ge \  
 E\left[ \frac{\left\{X_{\ve,{\rm class}}(\tau)-X^*_\ve(\tau)\right\}^2}{2(T-\tau)}+q_\ve (X^*_\ve(\tau),0) \ ;\big|  \ X^*_\ve(T)=x\right]  \ ,  
\ee
where $\tau$ is any stopping time satisfying $\tau\ge\tau^*_{\ve.x.T}\vee 0$. 
We use the fact that
\be \label{BK4}
 \inf_{\tau<T}\frac{X_{0,{\rm class}}(\tau)^2}{2(T-\tau)} \ = \  \inf_{\tau<T}\frac{\left\{x+k(T-\tau)\right\}^2}{2(T-\tau)} \ = \  2kx \ ,
\ee
where the minimizing $\tau$ satisfies $T-\tau=x/k$. We can obtain a lower bound on $q_\ve(x,T)$ from (\ref{BJ4}), (\ref{BK4}) if we show for any $\del>0$ that  with high probability
$X_{\ve,{\rm class}}(s)\ge (1-\del)X_{0,{\rm class}}(s), \ s\le T$, provided $X_{\ve,{\rm class}}(T)=X_{0,{\rm class}}(T)=x$ is sufficiently large.  Thus we need to compute the probability that 
\be \label{BL4}
X_{\ve,{\rm class}}(s)- (1-\del)X_{0,{\rm class}}(s) \ = \ \del x+\del k(T-s)+\sqrt{\ve}B(s) \ > 0 \quad {\rm for \ } s<T \ .
\ee
This probability is given by  $1$ minus the RHS of (\ref{AP4}) with $\del x$ replacing $x$ and $\del k$ replacing $k$, whence the probability of the event (\ref{BL4}) is 
$1-\exp[-2\del^2 kx /\ve]$.  Choosing $\tau=\tau^*_{\ve.x.T}\vee 0$ in (\ref{BJ4}), we conclude that
\begin{multline} \label{BM4} 
q_\ve(x,T) \ \ge \\
 \left\{1-\exp\left[-\frac{2\del^2 kx }{\ve}\right]\right\}\min\left[      2k(1-\del)^2x, \ \inf_{X>(1-\del)(x+kT), \ x'>0}  \left\{ \frac{ X-  x'\}^2}{2T}+q_\ve(x',0)\right\}                      \right] \ .
\end{multline}
Since we are assuming that $q_\ve(x',0)\ge f(x')$ and $\lim_{x'\ra\infty}f(x')=\infty$, we obtain on taking the limit $T\ra\infty$ in (\ref{BM4}) the inequality, 
\be \label{BN4}
\liminf_{T\ra\infty}\frac{q_\ve(x,T)}{x} \ \ge \ 2k(1-\del)^2 \left\{1-\exp[-2\del^2 kx /\ve]\right\}     \ .
\ee
The limit in (\ref{BN4}) is uniform for $0<x\le M, \ 0<\ve \le \ve_0$.  Evidently (\ref{BN4}) yields a lower bound close to $2k$ if $x/\ve>>1$. 

We wish to obtain a lower bound, which also holds as $x\ra 0$.  Let $\tau_{\ve,{\rm class},x,T}, \ s<T,$ be the first exit time from the half line $(0,\infty)$ for the diffusion
$X_{\ve,{\rm class}}(s), \ s<T,$ with $X_{\ve,{\rm class}}(T)=x$.  Since the function $x\ra q_\ve(x,T), \ x>0,$ is increasing, it follows from (\ref{BG4}) that 
$\tau^*_{\ve,x,T}>\tau_{\ve,{\rm class},x,T}$. We also have that
\be \label{BO4}
P(\tau_{\ve,{\rm class},x,T}<T-t) \ = \ \int_0^\infty G_{\ve,D}(x,x',t) \ dx' \ ,
\ee
where $G_{\ve,D}$ is given by (\ref{AO4}). Evidently the function $t\ra P(\tau_{\ve,{\rm class},x,T}<T-t)$ is decreasing, and from (\ref{AO4}), (\ref{BO4}) we see that
\be \label{BP4}
\lim_{t\ra\infty} P(\tau_{\ve,{\rm class},x,T}<T-t) \ = \ 1-\exp\left[-\frac{2 kx }{\ve}\right] \ , \quad t> 0 \ .
\ee
Note that (\ref{AP4}) is the integral of the pdf of  $\tau_{\ve,{\rm class},x,T}$, whence the RHS of (\ref{AP4}) and the RHS of (\ref{BP4}) add up to $1$. 
If $x/\ve\le K$ for some constant $K$, the asymptotic limit on the RHS of (\ref{BP4}) is closely approximated, uniformly for $0<\ve\le\ve_0$, when $t\ge N\ve$ where $N>>1$ is independent of $\ve$. The inequality (\ref{AT4}) gives a quantitative estimate of this. 

We assume $x$ satisfies $0<x\le Nk\ve/2$ with $N\ge 1$ and take $\tau=\tau^*_{\ve.x.T}\vee (T-N\ve)$ in (\ref{BJ4}). Using the quadratic inequality in (\ref{BJ4}) and the formula (\ref{BI4}),  we obtain a lower bound 
\begin{multline} \label{BQ4}
q_\ve(x,T) \ \ge \ \rho\left\{2kxP(\tau^*_{\ve,x,T}>T-N\ve)-2k\sqrt{\ve}E\left[B(T-N\ve); \ \tau^*_{\ve,x,T}<T-N\ve\right] \right\}\\
+E\left[ \frac{\left\{X_{\ve,{\rm class}}(T-N\ve)-X^*_\ve(T-N\ve)\right\}^2}{2N\ve}+q_\ve (X^*_\ve(T-N\ve),0) \ ; \  \tau^*_{\ve,x,T}<T-N\ve\right]   \ ,
\end{multline}
for any $\rho$ satisfying $0\le \rho\le 1$. 
In (\ref{BQ4}) we have used the identity $E[B(\tau)]=0$.  Observe from (\ref{BN4})  that for any $M>0$ there exists $N_0\ge 1,T_0>0$ such that 
\be \label{BR4}
 \inf_{X>Nk\ve/2, \ x'>0}  \left\{ \frac{ X-  x'\}^2}{2N\ve}+q_\ve(x',T-N\ve)\right\}   \ \ge \ \frac{Nk^2\ve}{32} \quad {\rm if \ } N\ge N_0, \ T\ge T_0, \  Nk\ve\le M \ .
\ee
We also have as in (\ref{BO4}) that
\begin{multline}  \label{BS4}
P\left(\tau_{\ve,{\rm class},x,T}<T-N\ve, \ X_{\ve,{\rm class}}(T-N\ve)<Nk\ve/2\right) \\
 = \ \int_0^{Nk\ve/2} G_{\ve,D}(x,x',N\ve) \ dx' \ \le \ \frac{Cx}{\ve}\exp\left[-\frac{Nk^2}{16}\right] \ ,
\end{multline}
for some constant $C$, independent of $x,\ve, N$ in the range $x>0, \ 0<\ve\le\ve_0, \ N\ge 1$.  Suppose now that
\be \label{BT4}
 P\left(\tau^*_{\ve,x,T}<T-N\ve\right) \ \ge \ \frac{2Cx}{\ve}\exp\left[-\frac{Nk^2}{16}\right] \ .
\ee
Setting $\rho=0$ in (\ref{BQ4}) and using the fact that $\tau^*_{\ve,x,T}>\tau_{\ve,{\rm class},x,T}$, we conclude from (\ref{BQ4})-(\ref{BS4}) that if (\ref{BT4}) holds then
\be \label{BU4}
q_\ve(x,T) \ \ge \  \frac{Nk^2\ve}{64}P\left(\tau^*_{\ve,x,T}<T-N\ve\right)  \quad {\rm if \ } N\ge N_0, \ T\ge T_0, \  Nk\ve\le M \ .
\ee
Evidently (\ref{BU4}) yields the lower bound on $q_\ve(x,T)$ unless
\be \label{BV4}
 P\left(\tau^*_{\ve,x,T}<T-N\ve\right) \ \le \frac{128x}{Nk\ve} \ .
\ee

We assume now that $0<x\le Nk\ve/2$ and that (\ref{BV4}) holds.  We apply the inequality (\ref{BQ4}) with $\rho=1$ and divide the second term on the RHS of (\ref{BQ4}) into the expectation over $B(T-N\ve)\le N^{2/3}\sqrt{\ve}$ and $B(T-N\ve)\ge N^{2/3}\sqrt{\ve}$.  Using (\ref{BI4})  we then have that
\begin{multline} \label{BW4}
q_\ve(x,T) \ \ge \ 2kx\left[1- \frac{128x}{Nk\ve}-\frac{128}{N^{1/3}k}\right] \\
-2kE\big[X_{\ve,{\rm class}}(T-N\ve)-X_{0,{\rm class}}(T-N\ve) \ ; \ \tau_{\ve,{\rm class},x,T}<T-N\ve, \\
 X_{\ve,{\rm class}}(T-N\ve)-X_{0,{\rm class}}(T-N\ve) > N^{2/3}\ve\big] \ .
\end{multline}
The expectation in (\ref{BW4}) is given by
\be \label{BX4}
 \int_{x+Nk\ve+N^{2/3}\ve}^\infty G_{\ve,D}(x,x',N\ve)[x'-(x+Nk\ve)] \ dx' \ \le \ Cx\exp\left[-N^{1/6}\right]  \ ,
\ee
where $C$ is a constant which can be chosen uniformly for $0<x\le Nk\ve/2, \ 0<\ve\le\ve_0, \ N\ge N_0$. 

To obtain the lower bound on $q_\ve(x,T)$ we first choose $N\ge N_0$ large and $M$ satisfying $Nk\ve_0= M$. Then (\ref{BN4}) holds and hence the lower bound up to an $O(1/N^{1/3})$ correction, when $Nk\ve/2\le x\le M$.  In the case $0<x\le Nk\ve/2$ the argument has just been given. 
   \end{proof}
  \begin{proposition}
Assume  the function $A(\cdot)$ satisfies the conditions of Proposition 5.1. Assume also there exists $\del_0,T_0>0$ such that
\be \label{BY4}
A\left(T+1/A(T)\right) \ \le \ \frac{A(T)}{1+\del_0} \quad {\rm if \ } T\ge T_0. 
\ee
 Letting $q_\ve(x,y,T)$ be the function defined by (\ref{A3}), (\ref{O3}), then for all $x,y>0$ one has $\lim_{T\ra\infty} \frac{q_\ve(x,y,T)}{2x}=1$. 
In addition the limit is uniform in any region $0<x\le M,  \ 0<\ve\le\ve_0,  \ y_0<y<y_\infty$, where $\ve_0, M>0$ and $0<y_0<y_\infty<\infty$. 
\end{proposition}
 \begin{proof}  In view of (\ref{AB4}) it is sufficient to establish a lower bound.  Since $\lim_{T\ra\infty}\la(x,y,T)=-1$ we may use a similar argument to the one just given of the lower bound in Proposition 5.2. We replace the assumption in Proposition 5.2 on the initial data by the following:
 \begin{multline} \label{BZ4}
 {\rm For \ any \ } \ve,\del, y_0>0,  \ {\rm there \ exists \ a \ function \ } f:[1,\infty)\ra\mathbb{R}  \  {\rm such \  that \ }  \\
\lim_{T\ra\infty} f(T)=\infty, \  \ q_\ve(\del/A(T),y,T)\ge f(T) \ {\rm for \ } 0<\ve\le\ve_0, \  y\ge y_0 \ .
 \end{multline}
  The extra assumption (\ref{BY4}) on $A(\cdot)$ is needed for the proof of (\ref{BZ4}).   
  
 To begin the proof of (\ref{BZ4}) we consider any $T_0\ge 1$ and define the stopping time $\tau_T$ by  $\tau_T=\tau^*_{\ve,x,T}\vee T_0$. We use the formula (\ref{AH4}), whence we have the inequality
 \begin{multline} \label{CA4}
q_\ve(x,y,T) \ \ge \ E\left[ q_\ve(X^*_\ve(T_0),y,T_0) ; \ \tau^*_{\ve,x,T}\le T_0\right] \\
+\frac{\sig^2_A(T)}{2}E\left[ \frac{\left\{X_{\ve,{\rm class}}(T_0,T)-X^*_\ve(T_0)\right\}^2}{\sig^2_A(T_0)\sig^2_A(T_0,T)} ; \ \tau^*_{\ve,x,T}\le T_0\right]  \ .
\end{multline}
 We see from (\ref{AF4}), (\ref{AI4}) that
 \be \label{CB4}
 {\rm Var}\left[\sqrt{\ve}\frac{\sig^2_A(s)}{m_{1,A}(s)}Z(s)\right] \ = \ \ve\frac{\sig_A^2(s)\sig_A^2(s,T)}{\sig^2_A(T)} \ \le \ \ve\sig^2_A(s) \ , \quad 0<s<T \ .
 \ee
 From Proposition 3.3 of \cite{cd} one has that the inequality (\ref{Z4}) extends to $\ve>0$, so 
 \be \label{CC4}
q_\ve(x,y,T) \ \ge \ \ \frac{2m_{1,A}(T)xy}{\sig^2_A(T)} \ , \quad x,y,T>0 \ .
\ee 
It follows then from (\ref{F4}), (\ref{AF4}), (\ref{CA4})-(\ref{CC4}) that for any $\rho,K>0$ there exists $T_{\rho,K}>T_0$ such that
\be \label{CD4}
q_\ve(x,y,T) \ \ge K \quad {\rm for \ }   \ y\ge y_0, \ \ 0<\ve\le\ve_0, \ T\ge T_{\rho,K}, \quad {\rm if \ } P(\tau^*_{\ve,x,T}\le T_0)\ge \rho \ .
\ee
From (\ref{CD4}) we may assume  in the remainder of the argument that $\tau^*_{\ve,x,T}> T_0$ with probability at least $1-\rho$, where $\rho>0$ may be taken arbitrarily small. 

Next we consider $T>1/A(0)$. Since the function $s\ra s+1/A(s), \ s>0,$ is strictly increasing, there exists unique  $\tilde{T}<T$ such that 
$\tilde{T}+1/A(\tilde{T})=T$. It follows from (\ref{BY4}) that $A(T)\le A(\tilde{T})/(1+\del_0)$. We have from (\ref{E4}), (\ref{F4}) that
 \begin{multline} \label{CE4}
 x_{\rm class}(s,T) \ \ge \ \frac{\sig_A^2(s)}{m_{1,A}(s)^2}
  \left( \frac{\sig_A^2(T)}{m_{1,A}(T)^2}   \right)^{-1}\frac{[x+m_{2,A}(s,T)]}{m_{1,A}(s,T)} \\
 - \  \frac{1}{m_{1,A}(s)}\frac{\sig_A^2(s,T)}{m_{1,A}(s,T)^2}
 \left(  \frac{\sig_A^2(T)}{m_{1,A}(T)^2}   \right)^{-1}\frac{m_{2,A}(s)}{m_{1,A}(s)} \  , \quad 0<s<T \ .
 \end{multline}
It follows from (\ref{A4}) (d),(e) and  (\ref{CE4}) that for any $\nu>0$ there exists $T_\nu>1/A(0)$ such that
\be \label{CF4}
 x_{\rm class}(\tilde{T},T) \ \ge \ \frac{1}{e(1+\nu)}\left[x+\frac{1}{A(\tilde{T})}\right] \quad {\rm if \ } x>0, \ T\ge T_\nu \ .
\ee
From (\ref{AF4}) we have that
\begin{multline} \label{CG4}
X_{\ve,{\rm class}}(s,T) = \ x_{\rm class}(s,T) - \sqrt{\ve} \ \Phi(s,T)\tilde{Z}(s,T) \ , \\
{\rm where \ } \Phi(s,T) \ = \ \frac{\sig^2_A(s)m_{1,A}(s,T)}{\sig^2_A(T)}  \ , \quad \tilde{Z}(s,T)  \ = \ \int_s^T\frac{dB(s')}{\Phi(s',T)}        \ .
\end{multline}
The function $s\ra \Phi(s,T)$ is positive increasing and $\Phi(T,T)=1$. We also have that
\be \label{CG*4}
\Phi(\tilde{T},T) \ = \  \frac{\sig^2_A(\tilde{T})m_{1,A}(\tilde{T},T)}{\sig^2_A(T)}  \ = \ \frac{\sig^2_A(\tilde{T})m_{1,A}(\tilde{T},T)}{m_{1,A}(\tilde{T},T)^2\sig^2_A(\tilde{T})+\sig^2_A(\tilde{T},T)}  \ .
\ee
Since $\sig^2_A(\tilde{T},T)\le m_{1,A}(\tilde{T},T)^2/A(\tilde{T})$ and $1\le m_{1,A}(\tilde{T},T)\le e$,  we can obtain a lower bound on $\Phi(\tilde{T},T)$, independent of  $T$ as $T\ra\infty$,  if we can bound above $1/\sig^2_A(\tilde{T})$ by a constant times $A(\tilde{T})$.  Note this is tantamount to the linear term in the formula (\ref{C4}) for the drift $x\ra\la(x,y,T)$ being dominated by $A(T)x$.  From (\ref{A4}) (e) we see that $\lim_{T\ra\infty} m_{1,A}(\tilde{T},T)\Phi(\tilde{T},T)=1$, whence the lower bound follows. 
From (\ref{AI4}) we see that
\be \label{CH4}
{\rm Var}[\tilde{Z}(s,T)] \ =  \ \frac{\sig_A^2(s,T)}{m_{1,A}(s,T)\Phi(s,T)}  \ .
\ee
We also have by the reflection principle that
\be \label{CI4}
P\left(\sup_{s<s'<T}\tilde{Z}(s',T)>a\right) \ =  \ 2 P\left(\tilde{Z}(s,T)>a\right) \ , \quad a>0 \ .
\ee
The identities (\ref{CG4})-(\ref{CI4}) enable us to bound below the diffusive path $X_{\ve,{\rm class}}(\cdot,T)$ with high probability.

Observe that the RHS of (\ref{E4}) is non-negative if $x,y\ge 0$ and $A(\cdot)$ is non-negative. Hence similarly to (\ref{CF4}), we may choose $T_\nu$ large enough so that
\be \label{CJ4}
 x_{\rm class}(s,T) \ \ge \ \frac{x}{e(1+\nu)} \quad {\rm if \ } x>0, \ T\ge T_\nu, \ \tilde{T}\le s\le T \ .
\ee
From (\ref{CH4}), (\ref{CI4}) and (\ref{I8}) we have that 
\be \label{CK4}
P\left(\sup_{\tilde{T}<s'<T}\tilde{Z}(s',T)>a\right) \ \le  \ \left(\frac{2}{\pi}\right)^{1/2}\frac{\sig_T}{a}\exp\left[-\frac{a^2}{2\sig_T^2}\right] \ , \quad \sig_T^2 \ = \ \frac{\sig^2_A(\tilde{T},T)}{m_{1,A}(\tilde{T},T)\Phi(\tilde{T},T)} \ .
\ee
We choose $T_\nu$ large enough so that
\be \label{CL4}
\sig^2_T \ \le \ \frac{e^2(1+\nu)}{A(\tilde{T})} \quad {\rm if \ } T\ge T_\nu \ ,
\ee
and define the event $\mathcal{E}_T$ by
\be \label{CM4}
\mathcal{E}_T \ = \ \left\{\sup_{\tilde{T}<s<T} \sqrt{\ve_0} \ \Phi(s,T)\tilde{Z}(s,T)<\frac{1}{A(\tilde{T})^{2/3}}\right\} \ .
\ee
Choosing $a\simeq 1/A(\tilde{T})^{2/3}$ in (\ref{CK4}) we see from (\ref{CL4}) and (\ref{A4})(a) there exists $T_1>1/A(0)$ such  that
\be \label{CN4}
P\left(\mathcal{E}_T\right) \ \ge \  1-\exp\left[-\frac{1}{A(\tilde{T})^{1/4}}\right] \  ,
\quad T\ge T_1 \ .
\ee
We conclude from (\ref{CF4}), (\ref{CG4}), (\ref{CJ4}), (\ref{CM4}) that for any $\del$ satisfying  $0<\del<1/3$, there exists $T_\del\ge T_1$ such that
\begin{multline} \label{CO4}
\inf_{\tilde{T}<s<T} X_{\ve,{\rm class}}(s,T) \ge \  \frac{\del}{3A(T)} \quad {\rm and \ }   X_{\ve,{\rm class}}(\tilde{T},T) \ge \  \frac{1}{3}\left[\frac{\del}{A(T)}+\frac{1}{A(\tilde{T})} \right] \\
{\rm on \ the \ event \ } \mathcal{E}_T \  {\rm when \ } \ x=\frac{\del}{A(T)} \ , \  T\ge T_\del \ .
\end{multline}
It follows from (\ref{CO4}) that
\begin{multline} \label{CP4}
 \frac{\sig^2_A(T)}{2}\frac{X_{\ve,{\rm class}}(s,T)^2}{\sig^2_A(s)\sig^2_A(s,T)} \ge 
    \frac{\del^2}{20A(T)} \ , \quad \frac{\sig^2_A(T)}{2}\frac{[X_{\ve,{\rm class}}(\tilde{T},T)-x']^2}{\sig^2_A(\tilde{T})\sig^2_A(\tilde{T},T)} \ \ge \ 
     \frac{\del^2}{20A(T)} \\
   {\rm on \ the \ event \ } \mathcal{E}_T \  {\rm when \ } \ x=\frac{\del}{A(T)}, \  T\ge T_\del, \quad {\rm if \ } \ \tilde{T} <s\le T, \ 0<x'<\frac{\del}{A(\tilde{T})} \ .
 \end{multline}
 Let $\Om_T$ be the event
 \be \label{CQ4}
 \Om_T \ = \ \left\{   x=\frac{\del}{A(T)}, \   \tilde{T}<\tau^*_{\ve,x,T}<T\right\}\cup  \left\{   x=\frac{\del}{A(T)}, \   \tau^*_{\ve,x,T}\le \tilde{T},  \ 0<X^*_\ve(\tilde{T})<\frac{\del}{A(\tilde{T})} \right\} \ .
 \ee
We conclude from (\ref{AH4}), (\ref{CP4}), (\ref{CQ4}) that
 \be \label{CR4}
 q_\ve(\del/A(T),y,T) \ \ge \  P(\Om_T^c\cap\mathcal{E}_T) q_\ve(\del/A(\tilde{T}),y,\tilde{T})+  P(\Om_T\cap\mathcal{E}_T) \frac{\del^2}{20A(T)}  \ .
 \ee
where $\Om_T^c$ is the complement of $\Om_T$.

  We choose any $\del$ with $0<\del<1/3$ and $T_\del$ as in (\ref{CO4}).   For $T>T_\del+1/A(T_\del)$ we may define a sequence of times $T_n, \ n=1,2,.., N,$ with  $T_\del\le T_1< T_\del+1/A(T_\del), \ T_N=T$, and $T_{n+1}=T_n+1/A(T_n), \ n=1,2,..,N-1$.   We  define events $\Om_n, \ n=1,2,..,N-1,$ by
 \begin{multline} \label{CS4}
 \Om_n \ = \ \{T_{n+1}\ge\tau^*_{\ve,x,T}>T_n, \ X^*_\ve(T_m)\ge \del/A(T_m),  \ m=n+1,\dots, N\} \\
 \cup\{\tau^*_{\ve,x,T}\le T_n, \ X^*_\ve(T_n)<\del/A(T_n)\} \  ,
 \end{multline} 
where $x=\del/A(T)$.  The events $\Om_n$ are disjoint and
\be \label{CT4}
 \sum_{n=1}^{N-1}P\left(\Om_n              \right)+P\left(  \tau^*_{\ve,x,T}\le T_1, \ X^*_\ve(T_m)\ge \del/A(T_m), \ m=1,\dots N              \right) \ = \ 1 \ .
\ee
We define the probabilities 
\be \label{CU4}
p_{j-1} \ = \ \frac{P(\Om_{T_j}\cap\mathcal{E}_{T_j})}{P(\mathcal{E}_{T_j})}  \ , \quad \la_{j-1}=P(\mathcal{E}_{T_j}) \ , \ j=2,\dots, N  \ .
\ee
Then (\ref{CR4}) implies that
\be \label{CV4}
 q_\ve(\del/A(T_n),y,T_n) \ \ge \  \la_{n-1}\left[(1-p_{n-1}) q_\ve(\del/A(T_{n-1}),y,T_{n-1})+  p_{n-1}\frac{\del^2}{20A(T_n)}\right]  \ , \  \ n=2,\dots, N.
 \ee
 Iterating the inequality (\ref{CV4}) starting with $n=N$ down to $n=2$, we conclude that
 \be \label{CW4}
  q_\ve(\del/A(T),y,T) \ \ge \ \prod_{n=1}^{N-1}\la_n\left[1-\prod_{n=1}^{N-1}(1-p_j)\right] \frac{\del^2}{20A(T_\del)} \ .
 \ee 
 We also have that  $P(\Om_{T_n})\ge P(\Om_{n-1}), \ n=2,\dots, N$, whence we see that
 \be \label{CX4}
 \sum_{n=1}^{N-1} p_n \ \ge \  \sum_{n=2}^{N} P(\Om_{T_n})-\sum_{n=2}^{N} P(\mathcal{E}^c_{T_n}) \ \ge \ 
  \sum_{n=1}^{N-1} P(\Om_n)-\sum_{n=2}^{N} P(\mathcal{E}^c_{T_n}) \ .
 \ee
 It follows from (\ref{BY4}), (\ref{CN4}) that $T_1$ in (\ref{CN4}) may be chosen so that
 \be \label{CY4}
 \sum_{n=2}^{N} P(\mathcal{E}^c_{T_n}) \ \le \  \frac{1}{4}, \quad \prod_{n=1}^{N-1}\la_n \ \ge \frac{1}{2} \quad {\rm if \ } T_\del\ge T_1 \ .
 \ee
 Suppose now that
 \be \label{CZ4}
  \sum_{n=1}^{N-1} P(\Om_n) \ \ge \ \frac{1}{2} \ .
 \ee
 It follows then from (\ref{CX4})-(\ref{CZ4}) that
 \be \label{DA4}
 \prod_{n=1}^{N-1}(1-p_j) \ \le \ \exp\left[-\sum_{n=1}^{N-1}p_j\right] \ \le \ \exp\left[-\frac{1}{4}\right] \ .
 \ee
 We conclude from (\ref{CW4}), (\ref{CY4}), (\ref{DA4}) that if (\ref{CZ4}) holds then
 \be \label{DB4}
  q_\ve(\del/A(T),y,T) \ \ge \ \frac{1}{2}\left[1-e^{-1/4}\right] \frac{\del^2}{20A(T_\del)} \ .
 \ee
 
 To prove (\ref{BZ4}) we observe first that if (\ref{CZ4}) holds then  $  q_\ve(\del/A(T),y,T) $  is bounded below by the RHS of (\ref{DB4}) for $T>T_\del+1/A(T_\del)$. 
 This bound is uniform for all $\ve,y$ satisfying $0<\ve\le\ve_0, \ y>0$.  In the case when (\ref{CZ4}) does not hold we have from (\ref{CT4}) that
 \be \label{DC4}
 P\left(  \tau^*_{\ve,x,T}\le T_1 \       \right) \ \ge \ \frac{1}{2} \ , \quad x=\frac{\del}{A(T)} \ .
\ee
Then (\ref{CD4}), (\ref{DC4}) imply a lower bound on  $  q_\ve(\del/A(T),y,T) $ for large $T$ provided $0<\ve\le\ve_0, \ y\ge y_0$.   Since $K$ in (\ref{CD4}) may be chosen arbitrarily large, and $A(T_\del)$ in (\ref{DB4}) arbitrarily small, the lower bound (\ref{BZ4}) follows. 

We proceed now as in the proof of the lower bound in Proposition 5.2 beginning at (\ref{BG4}), replacing the condition on the initial  data in Proposition 5.2 by 
(\ref{BZ4}).  Thus we consider $q_\ve(x,y,T)$  for large $T$ and use (\ref{BZ4}) at time $\tilde{T}$.  We take $\tau_T=\max\{\tau^*_{x,\ve,T},\tilde{T}\}$ in (\ref{AH4}). 
It follows from (\ref{A4}) (d),(e), (\ref{E4}), (\ref{F4}) that for any $\nu>0$ there exists $T_\nu>1/A(0)$ such that if $T\ge T_\nu$ then
\be \label{DD4}
x_{\rm class}(s,T) \ \ge \ \frac{1}{(1+\nu)m_{1,A}(s,T)} \left[x+m_{2,A}(s,T)\right] \quad {\rm for \ } \tilde{T}\le s\le T \ .
\ee
We wish next to estimate from below the probability that $\tau^*_{\ve,x,T}<\tilde{T}$. To see this we first observe from (\ref{AF4}) that
\be \label{DE4}
X_{\ve,{\rm class}}(s,T)>0 \quad {\rm \ if \ } \frac{\sig_A^2(T)}{m_{1,A}(T)^2}
 \left(  \frac{\sig_A^2(s)}{m_{1,A}(s)^2}   \right)^{-1}m_{1,A}(s,T) x_{\rm class}(s,T)-\sqrt{\ve}\frac{\sig^2_A(T)}{m_{1,A}(T)}Z(s)>0 \ .
\ee
From (\ref{AI4}) we have that
 \be \label{DF4}
 {\rm Var}\left[\frac{\sig^2_A(T)}{m_{1,A}(T)}Z(s)\right] \ = \ \frac{\sig_A^2(T)}{m_{1,A}(T)^2} \left(  \frac{\sig_A^2(s)}{m_{1,A}(s)^2}   \right)^{-1}\sig^2_{A}(s,T)  \ , \quad 0<s<T \ .
 \ee
 Since the function $s\ra \sig^2_A(s)/m_{1,A}(s)^2$ is increasing, we have from  (\ref{DD4}), (\ref{DF4}) the inequality
 \begin{multline} \label{DG4}
  \frac{\sig_A^2(T)}{m_{1,A}(T)^2}\left(  \frac{\sig_A^2(s)}{m_{1,A}(s)^2}   \right)^{-1}m_{1,A}(s,T) x_{\rm class}(s,T) \\
  \ge \ \frac{x}{1+\nu}+ \frac{1}{(1+\nu)m_{1,A}(s,T)}{\rm Var}\left[\frac{\sig^2_A(T)}{m_{1,A}(T)}Z(s)\right]  \ , \quad s<T \ .
 \end{multline} 
Using the fact that the martingale $s\ra Z(s)$ is a rescaled Brownian motion,  and the inequality $m_{1,A}(s,T)\le e$ for $\tilde{T}\le s\le T$, we conclude from (\ref{AP4})  with $x$ replaced by $x/(1+\nu)$ and $k=1/(1+\nu)e$ and (\ref{DD4})-(\ref{DF4})  that
\be \label{DH4}
P\left(\tau_{\ve,{\rm class},x,T}<\tilde{T}\right) \ \ge \ 1-\exp\left[    -\frac{2x}{\ve(1+\nu)^2e}                   \right] \quad {\rm if \ } T\ge T_\nu \ ,
\ee
where $\tau_{\ve,{\rm class},x,T}$ is the first exit time from the half line $(0,\infty)$ for $X_{\ve,{\rm class}}(s,T), \ s<T$. 

We use (\ref{DH4}) to prove the analogue of (\ref{BM4}).  Thus from (\ref{DH4}) we have for $0<\del<1$ that
\begin{multline} \label{DI4}
X_{\ve,{\rm class}}(s,T) \ >  \ (1-\del) x_{\rm class}(s,T), \ \ \tilde{T}\le s\le T, \\
{\rm with \ probability \ at \ least \ } 1-\exp\left[    -\frac{2\del^2x}{\ve(1+\nu)^2e}  \right] \ .
\end{multline}
We also have from (\ref{DD4}) that for any $M>0$ then
\be \label{DJ4}
\frac{\sig^2_A(T)}{2}\inf_{\tilde{T}<\tau<T} \frac{x_{\rm class}(\tau,T)^2}{\sig^2_A(\tau)\sig^2_A(\tau,T)} \ \ge \  \frac{2x}{(1+\nu)^2}, \quad 0<x\le M,  \  T\ge T_{\nu,M}  \ ,
\ee
where $T_{\nu,M}\ge T_\nu$ depends also on $M$.  From (\ref{DD4}) we see there are constants $C_1,C_2>0$ such that 
\begin{multline} \label{DK4}
\inf_{X>(1-\del) x_{\rm class}(\tilde{T},T)}\left[ \frac{\sig^2_A(T)}{2} \frac{(X-x')^2}{\sig^2_A(\tilde{T})\sig^2_A(\tilde{T},T)} +q_\ve(x',y,\tilde{T}) \right]  \\
\ge \ \min\left[ \frac{\sig^2_A(T)}{8} \frac{(1-\del)^2x_{\rm class}(\tilde{T},T)^2}{\sig^2_A(\tilde{T})\sig^2_A(\tilde{T},T)}, \ q_\ve\left(\frac{1}{2}(1-\del) x_{\rm class}(\tilde{T},T),y,\tilde{T}\right)\right]  \\
\ge \min\left[\frac{C_1}{A(\tilde{T})}, \ q_\ve\left(\frac{C_2}{A(\tilde{T})},y,\tilde{T}\right)\right] \quad {\rm if \ } 0<\del<\frac{1}{2}, \ x'>0, \ T\ge T_\nu \ .
\end{multline}
It follows from (\ref{AH4}),  (\ref{BZ4}) and (\ref{DI4})-(\ref{DK4}) that
\be \label{DL4}
\liminf_{T\ra\infty}\frac{q_\ve(x,y,T)}{2x} \ \ge \ \frac{(1-\del)^2}{(1+\nu)^2}\left\{1-\exp\left[    -\frac{2\del^2x}{\ve(1+\nu)^2e}  \right]\right\} \ .
\ee
The limit in (\ref{DL4}) is uniform for $0<x\le M, \ 0<\ve\le\ve_0,  \ y_0\le y\le y_\infty$. 

The remainder of the proof follows the same lines as the proof of the lower bound in Proposition 5.2, beginning after (\ref{BN4}). 
\end{proof}  
  \vspace{.1in}
  
  \section{Convergence to the exponential distribution}
Here we complete the proof of Theorem 1.1. First we extend the method used in proving Proposition 4.1 to prove an analogous result for the half line problem. 
\begin{proposition}
Assume $c_\ve(x,t), \ x,t>0,$ and $X_{\ve,t}, \ t>0,$ are as in Lemma 4.2.   Then 
 \be \label{A5}
 \frac{X_{\ve,t}}{\langle X_{\ve,t}\rangle} \xrightarrow{D} \mathcal{X} \ , \quad  \ {\rm as \ \ } t\ra\infty \ .
 \ee
Let $g:[0,\infty)\ra\mathbb{R}^+$ be a continuous function which satisfies $\lim_{t\ra\infty} g(t)=\infty$.  Then one has
 \be \label{B5}
 \lim_{t\ra\infty}\sup_{x>g(t)}|\beta_{X_{\ve,t}}(x)-1| \ = \ 0  .
  \ee
\end{proposition}
\begin{proof}
We first show that (\ref{AN2}) holds for the half line problem.  To do this we write $\La_\ve(T)$ as the ratio of (\ref{D3}) to (\ref{C3}).  Note that Lemma 4.1 implies that (\ref{A4}) holds for $A(\cdot)=1/\La_\ve(\cdot)$, and Proposition 4.2 implies that (\ref{BY4}) also holds when  $A(\cdot)=1/\La_\ve(\cdot)$. Hence the conclusion of Proposition 5.3 holds when $A(\cdot)=1/\La_\ve(\cdot)$. It follows then from (\ref{E3}), Lemma 4.1 and Proposition 5.3 that
\be \label{C5}
\lim_{T\ra\infty} \frac{m_{2,A}(T)E[X_{\ve,y,T}]}{\ve \sig^2_A(T)} \ = \ 1 \ ,
\ee
with the limit in (\ref{C5})  being uniform in $y$ for $y$ in any interval $0<y_0<y<y_\infty$. The function $(y,t)\ra u_\ve(y,t,T), \ y>0,  \ t<T,$ defined by (\ref{B3}) is the solution to the terminal value problem
\be \label{D5}
\frac{\pa u_\ve(y,t)}{\pa t}+[A(t)y-1]\frac{\pa u_\ve(y,t)}{\pa y}+\frac{\ve}{2}\frac{\pa^2 u_\ve(y,t)}{\pa y^2} \ = \ 0, \quad y>0, \ t<T,
\ee
\be \label{E5}
u_\ve(y,T) \ = \ u_T(y), \quad y\in\R \ ,
\ee
with zero Dirichlet condition $u_\ve(0,t)=0, \ t<T$, and terminal condition $u_\ve(\cdot,T)\equiv 1$. By the maximum principle \cite{pw} we see that for any $t<T$ the function $y\ra u_\ve(y,t,T)$ is increasing. Now (\ref{AN2}) follows from (\ref{G3}), (\ref{C5}). 

To prove (\ref{A5}) we use a similar identity to (\ref{AF2}), 
\begin{multline} \label{F5}
P\left(\frac{X_{\ve,T}}{\langle X_{\ve,T} \rangle}> x\right) \ 
 =  \\ \int_0^{y_\infty} P\left(\frac{X_{\ve,y,T}}{\langle X_{\ve,T} \rangle}> x\right) u_\ve(y,0,T)c_\ve(y,0) \ dy 
  \Bigg / \int_0^{y_\infty} u_\ve(y,0,T)c_\ve(y,0) \ dy  \ .
\end{multline}
We have for $x>0$ that
\begin{multline} \label{G5}
P\left(X_{\ve,y,T} >x\right) \ = \\
 \frac{E[K_{\ve,D}( \sqrt{\ve}\sig_A(T)\{Z-z_{y,T}/\sqrt{\ve}\},y,0,T)H(\sqrt{\ve}\sig_A(T))\{Z-z_{y,T}/\sqrt{\ve}\}-x \ | \ Z>z_{y,T}/\sqrt{\ve}]}
 {E[K_{\ve,D}( \sqrt{\ve}\sig_A(T)\{Z-z_{y,T}/\sqrt{\ve}\},y,0,T) \ | \ Z>z_{y,T}/\sqrt{\ve}]} \ ,
\end{multline}
where $H:\mathbb{R}\ra\mathbb{R}$ is the Heaviside function. As with proving the limit (\ref{C5}), we conclude from (\ref{AN2}), (\ref{G5}) and Proposition 5.3 that
\be \label{H5}
\lim_{T\ra\infty} P\left(\frac{X_{\ve,y,T}}{\langle X_{\ve,T} \rangle}> x\right) \ = \ e^{-x} \quad {\rm for \ } x>0 \ ,
\ee
and the limit is uniform in any interval $0<y_0<y<y_\infty$.  The convergence in distribution (\ref{A5}) follows from (\ref{F5}), (\ref{H5}) upon using the monotonicity of the function $y\ra u_\ve(y,0,T)$ again. 

 To begin the proof of (\ref{B5}) we first note we cannot set $g(\cdot)\equiv0$ as in Proposition 2.1 since the zero Dirichlet boundary condition implies that $\beta_{X_{\ve,T}}(0)=0$ for $T>0$.  Similarly to  (\ref{AG2}), (\ref{AH2}) we observe that $\beta_{X_{\ve,T}}(x)=A_{\ve,D}(x,T)C_{\ve,D}(x,T)/B_{\ve,D}(x,T)^2$. The functions $A_{\ve,D} ,B_{\ve,D}$ are as in (\ref{AD3}), (\ref{AE3}), while $C_{\ve,D}$ is given by the formula
\begin{multline} \label{I5}
C_{\ve,D}(x,T) \ = \ \int_0^{y_\infty}dy \int_0^\infty dx'  \ K_{\ve,D}(x+x',y,0,T) \\
\times x'\exp\left[\frac{b(T)(x+x')y}{\ve\sig_A(T)}-\frac{a(T)x'}{\ve\sig_A(T)}-\frac{x'(2x+x')}{2\ve\sig^2_A(T)}\right]\tilde{c}_\ve(y,0) \ .
\end{multline}
Comparing (\ref{AD3}), (\ref{AE3}), (\ref{I5}) to (\ref{AG2}), (\ref{AH2}) we see from Proposition 5.3 and the argument  of Proposition 3.1 that for any $\del>0$ there exists $x_\del,T_\del>0$ such that $\sup_{x\ge x_\del}|\beta_{X_{\ve,T}}(x)-1|<\del$ for $T\ge T_\del$. Note that to conclude this we use  the fact that the function $y\ra\exp[b(T)xy/\sqrt{\ve}\sig_A(T)]$ is increasing for all $x>0$. The limit (\ref{B5}) evidently follows.   
\end{proof}
\begin{proposition}
Assume $c_\ve(x,t), \ x,t>0,$ and $X_{\ve,t}, \ t>0,$ are as in Lemma 4.2.   Then 
\be \label{J5}
\lim_{t\ra\infty} \frac{d}{dt}  \langle X_{\ve,t}\rangle \ = \ 1 \ .
\ee
\end{proposition}
\begin{proof}
We proceed similarly to the proof of Proposition 4.1, whence  we may assume that (\ref{M3}) holds. Also from (\ref{B5}) of Proposition 6.1 we may additionally assume that
\be \label{K5}
\ve<\nu \quad {\rm and \ } \sup_{x>\del}|\beta_{X_{\ve,0}}(x)-1|\le \rho,
\ee
where $\rho,\nu,\del>0$ may be chosen arbitrarily small.  As in (\ref{I2}) we have that
\begin{multline} \label{L5}
\int_0^\infty dy \int_0^\infty dx  \ G_\ve(x,y,0,T)c_\ve(y,0)  \\
 = \ \int_0^\infty dy \  P\left(Z> \frac{m_{2,A}(T)-m_{1,A}(T)y}{\sqrt{\ve}\sig_A(T)}\right) \ c_\ve(y,0) \ = \ I_\ve(T) \ ,
\end{multline}
where $A(\cdot)=1/\La_\ve(\cdot)$ is decreasing and $A(0)=1$.  We define the function $w_\ve(\cdot)$ by
\be \label{P5}
w_\ve(y) \ = \ \int_y^\infty c_\ve(y',0) \ dy' \quad {\rm for \ } y\ge 0 \ .
\ee
The integral  on the RHS of (\ref{L5}) is bounded above as
\be \label{M5}
I_\ve(T) \ \le \  P\left(Z>\frac{1}{\ve^{1/4}}\right)w_\ve(0)+w_\ve\left(\al(T)-\ve^{1/4}\beta(T)\right) \ ,
\ee
where
\be \label{N5}
\al(T) \ =  \ \frac{m_{2,A}(T)}{m_{1,A}(T)} \ , \quad \beta(T) \ = \ \frac{\sig_A(T)}{m_{1,A}(T)} \ .
\ee
Similarly we have a lower bound
\be \label{O5}
I_\ve(T) \ \ge \  \left[1-P\left(Z>\frac{1}{\ve^{1/4}}\right)\right]w_\ve\left(\al(T)+\ve^{1/4}\beta(T)\right) \ .
\ee

We obtain using (\ref{M3}) bounds for $I_\ve(T)$ in terms of $w_\ve(\al(T))$.    To see this first observe from (\ref{W2})-(\ref{Z2})  the identity
\be \label{Q5}
w_\ve(x) \ = \ \frac{w_\ve(0) v_\ve(x)}{v_\ve(0)}\exp\left[-\int_0^x v_\ve(x') \ dx'\right] \ , \quad x> 0 \ .
\ee
From (\ref{Y2}), (\ref{M3}) we see that $v_\ve(\cdot)$ has the properties
\be \label{R5}
1-C_1 \ \le \ \frac{1}{v_\ve(x)^2}\frac{dv_\ve(x)}{dx} \ \le  \ 1 \ , \ \ x>0, \quad v_\ve(0) \ = \ 1 \ ,
\ee
where $C_1$ is the constant in (\ref{M3}).  It follows from (\ref{R5}) that
\be \label{S5}
v_\ve(x) \  \le  \ 2, \quad \left|\frac{dv_\ve(x)}{dx}\right| \ \le \  4(C_1+1) \ ,  \quad {\rm for \ } 0<x<1/2 \ .
\ee
Choosing $x_0=1/8(C_1+1)$, we have from (\ref{R5}), (\ref{S5}) that $1/2\le v_\ve(x)\le 3/2$ for $0\le x\le x_0$.  Applying (\ref{Q5}), (\ref{S5}) to (\ref{M5}), (\ref{O5}) we have that
\begin{multline} \label{T5}
 \left[1-P\left(Z>\frac{1}{\ve^{1/4}}\right)\right]\left\{1-8(C_1+1)\ve^{1/4}\beta(T)\right\}\exp\left[-3\ve^{1/4}\beta(T)/2\right]
 \\
 \le \ \frac{I_\ve(T)}{w_\ve\left(\al(T)\right)} \ \le \ 2e^{3x_0/2}P\left(Z>\frac{1}{\ve^{1/4}}\right)+\left\{1+8(C_1+1)\ve^{1/4}\beta(T)\right\}\exp\left[3\ve^{1/4}\beta(T)/2\right] \ ,
\end{multline}
provided $\al(T)+\ve^{1/4}\beta(T)\le x_0$. 

A similar argument may be made to estimate the integral
\be \label{U5}
J_\ve(T) \ = \ \int_0^\infty dy \int_0^\infty dx  \ x \ G_\ve(x,y,0,T)c_\ve(y,0) \ 
\ee
in terms of the function $h_\ve(\cdot)$ defined by 
\be \label{V5}
h_\ve(y) \ = \ \int_x^\infty w_\ve(y') \ dy' \ , \quad y\ge 0.
\ee
To do this we first observe that
\begin{multline} \label{W5}
\int_0^\infty dx  \ x \ G_\ve(x,y,0,T) \ = \ m_{1,A}(T)\left[y-\al(T)\right] P\left(Z> \frac{m_{2,A}(T)-m_{1,A}(T)y}{\sqrt{\ve}\sig_A(T)}\right) \\
+ \left(\frac{\ve \sig^2_A(T)}{2\pi}\right)^{1/2}\exp\left[-\frac{\{m_{2,A}(T)-m_{1,A}(T)y\}^2}{2\ve\sig_A^2(T)}\right] \ .
\end{multline}
We conclude from (\ref{U5})-(\ref{W5}) that
\be \label{X5}
J_\ve(T) \ \le \  \left(\frac{\ve \sig^2_A(T)}{2\pi}\right)^{1/2}w_\ve(0)+m_{1,A}(T)h_\ve\left(\al(T)\right) \ .
\ee
Similarly to (\ref{O5}) we also have from (\ref{W5})  a lower bound
\be \label{Y5}
J_\ve(T) \ \ge \  \left[1-P\left(Z>\frac{1}{\ve^{1/4}}\right)\right]m_{1,A}(T)h_\ve\left(\al(T)+\ve^{1/4}\beta(T)\right)  \ .
\ee
Assuming again that  $\al(T)+\ve^{1/4}\beta(T)\le x_0$, we bound $J_\ve(T)$ using (\ref{Z2}), (\ref{S5}) as
\begin{multline} \label{Z5}
 \left[1-P\left(Z>\frac{1}{\ve^{1/4}}\right)\right]\exp\left[-3\ve^{1/4}\beta(T)/2\right]
 \\
 \le \ \frac{J_\ve(T)}{m_{1,A}(T)h_\ve\left(\al(T)\right)} \ \le \  3e^{3x_0/2}\left(\frac{\ve \sig^2_A(T)}{2\pi m_{1,A}(T)^2}\right)^{1/2}+1 \ .
\end{multline}

Next as in (\ref{P3})  we use the formula
\be \label{AA5}
\frac{\ve}{2}\frac{\pa c_\ve(0,T)}{\pa x} \ = \ \frac{1}{2}\int_0^{\infty} \frac{\pa q_\ve(0,y,T)}{\pa x} G_\ve(0,y,0,T)c_\ve(y,0) \ dy \ 
\ee
to estimate the LHS of (\ref{AA5}) in terms of $c_\ve(\al(T),0)$. We write the RHS of (\ref{AA5})  as a sum of the integral over the interval $\al(T)-\ve^{1/4}\beta(T) < y<\al(T)+\ve^{1/4}\beta(T)$ and the integral over the complement of this interval in $\mathbb{R}^+$. From Proposition 5.1 of \cite{cd} this latter integral is bounded above by
\be \label{AB5}
\frac{1}{\sqrt{2\pi\ve\sig^2_A(T)}} \exp\left[-\frac{1}{2\ve^{1/2}}\right]\left[w_\ve(0)+\frac{m_{1,A}(T)}{\sig^2_A(T)}h_\ve(0)\right] \ .
\ee
Again using Proposition 5.1. of \cite{cd}, the former integral is bounded above by
\be \label{AC5}
\frac{1}{m_{1,A}(T)}\left[1+\frac{\ve^{1/4}}{\sig_A(T)}\right]\sup_{|y-\al(T)|<\ve^{1/4}\beta(T)} c_\ve(y,0) \  .
\ee
To obtain a lower bound on the RHS of (\ref{AA5}) we use Proposition 6.1 of \cite{cd}. Thus since $\sup A(\cdot)\le 1$,  there exist universal constants 
$C_1,C_2>0$ such that 
\be \label{AD5}
\frac{1}{2} \frac{\pa q_\ve(0,y,T)}{\pa x} \ \ge \  1+\frac{m_{1,A}(T)[y-\al(T)]}{\sig^2_A(T)}-\frac{C_1\ve T^2}{y^2} \ ,
\ee
provided $0<T\le 1,  \ y\ge C_2T^2, \ \ve\le T^3$. 
Similarly to (\ref{AC5}) one obtains from (\ref{AD5}) the lower bound
\be \label{AE5}
\frac{1}{m_{1,A}(T)} \left[1-P\left(|Z|>\frac{1}{\ve^{1/4}}\right)- \frac{\ve^{1/4}}{\sig_A(T)}    -\frac{C_1\ve T^2}{\{\al(T)-\ve^{1/4}\beta(T)\}^2}\right]\inf_{|y-\al(T)|<\ve^{1/4}\beta(T)} c_\ve(y,0) \  .
\ee
We may bound $c_\ve(\cdot,0)$ in terms of the beta function $\beta_{X_{\ve,0}}(\cdot)=c_\ve(\cdot.0)h_\ve(\cdot)/w_\ve(\cdot)^2$ of (\ref{X2}).  Thus from (\ref{Z2}), (\ref{Q5}), (\ref{S5}) we have
\begin{multline} \label{AF5}
\sup_{|y-\al(T)|<\ve^{1/4}\beta(T)} c_\ve(y,0) \ \le \  \left\{1+8(C_1+1)\ve^{1/4}\beta(T)\right\}^2 \\
\times \exp\left[9\ve^{1/4}\beta(T)/2\right]\frac{w_\ve(\al(T))^2}{h_\ve(\al(T))}\sup_{|y-\al(T)|<\ve^{1/4}\beta(T)} \beta_\ve(y,0) \ ,
\end{multline}
provided $\al(T)+\ve^{1/4}\beta(T)\le x_0$. Similarly one obtains a lower bound
\begin{multline} \label{AG5}
\inf_{|y-\al(T)|<\ve^{1/4}\beta(T)} c_\ve(y,0) \ \ge \  \left\{1-8(C_1+1)\ve^{1/4}\beta(T)\right\}^2 \\
\times \exp\left[-9\ve^{1/4}\beta(T)/2\right]\frac{w_\ve(\al(T))^2}{h_\ve(\al(T))}\inf_{|y-\al(T)|<\ve^{1/4}\beta(T)} \beta_\ve(y,0) \ .
\end{multline}

We choose now $T_0,\nu>0$  such that
\be \label{AH5}
0<T_0\le 1, \quad 2C_2T_0^2 \ \le \ \al(T_0)\le \frac{x_0}{2} \ ,  \quad \nu^{1/4}\beta(T_0) \ \le \ \ \frac{\al(T_0)}{2} \ .
\ee
Note that the choice in (\ref{AH5}) is possible since $\lim_{T\ra 0}\al(T)/T=1$.  The inequalities (\ref{AF5}), (\ref{AG5}) both hold for $T=T_0$ and $\ve$ satisfying (\ref{K5}). From (\ref{B1}), (\ref{D1}) we have that
\be \label{AI5}
\frac{d}{dt}\langle X_{\ve,t}\rangle \Big|_{t=T_0} \ = \ \frac{\ve}{2}\frac{\pa c_\ve(0,T_0)}{\pa x} \frac{J_{\ve,D}(T_0)}{I_{\ve,D}(T_0)^2}\  , 
\ee
where $I_{\ve,D}(T), \ J_{\ve,D}(T)$ are defined as in (\ref{L5}), (\ref{U5}), but with the half line Dirichlet Green's function $G_{\ve,D}$ replacing the whole line Green's function $G_\ve$.  Evidently we have that $I_{\ve,D}(T)\le I_\ve(T),  \ J_{\ve,D}(T)\le J_\ve(T)$, whence (\ref{T5}), (\ref{Z5}) yield upper bounds on  $I_{\ve,D}(T), \ J_{\ve,D}(T)$. We can also see from the lower bound on $q_\ve(x,y,T)$ in Proposition 3.3 of \cite{cd} and (\ref{O3}) that the lower bounds in  (\ref{T5}), (\ref{Z5}) also hold with $T=T_0$ and $\nu$ small, up to  a multiplicative factor close to $1$. We conclude that the LHS of (\ref{AI5}) is equal to $1$ modulo terms in the parameters $\nu,\rho$ of (\ref{K5}) which converge to zero as $\nu,\rho\ra 0$.  Now (\ref{J5}) follows by arguing as in the proof of Lemma 7.3 of \cite{cdw}. 
\end{proof}
\appendix

\section{Properties of Gaussian Conditional Variables}
Let $Z$ be the standard normal variable and for $z\in\mathbb{R}$ let $X_z$ be the random variable $Z-z$ conditioned on $Z>z$.  Here we derive some properties of the variables $X_z$.
\begin{lem}
Let $m:\mathbb{R}\ra\mathbb{R}$ be the function $m(z)=\langle X_z\rangle, \ z\in\mathbb{R}$.  Then $m(\cdot)$ is a continuous positive decreasing function satisfying the inequalities,
\be \label{A8}
\frac{1}{z}-\frac{2}{z^3} \ <  \ m(z) \ < \ \frac{1}{z} \quad {\rm for \ } z>1 \ ,
\ee
\be \label{B8}
\max\{|z|, \sqrt{2/\pi}\}\ <  \ m(z) \ < \  \sqrt{2/\pi}+|z| \quad {\rm for \ } z<0 \ .
\ee
\end{lem} 
\begin{proof}
We have that
\begin{multline} \label{C8}
m(z) \ = \ \int_z^\infty (z'-z)e^{-z'^2/2} \ dz' \Big/ \int_z^\infty e^{-z'^2/2} \ dz'  \\
= \ e^{-z^2/2} \Big/ \int_z^\infty e^{-z'^2/2} \ dz' -z \ = \ \left[\int_0^\infty e^{-z'^2/2-zz'} \ dz'\right]^{-1}-z \ .
\end{multline} 
Differentiating the last formula on the RHS of (\ref{C8}) we see that
\be \label{D8}
\frac{dm(z)}{dz} \ = \ \int_0^\infty z'e^{-z'^2/2-zz'} \ dz'\Big/ \left[\int_0^\infty e^{-z'^2/2-zz'} \ dz'\right]^2-1 \ .
\ee
Using the fact that
\be \label{E8}
m(z) \ = \ \int_0^\infty z'e^{-z'^2/2-zz'} \ dz'\Big/ \int_0^\infty e^{-z'^2/2-zz'} \ dz' \ ,
\ee
we conclude that $m(\cdot)$ is the solution to the Riccati equation
\be \label{F8}
\frac{dm(z)}{dz} \ = \  m(z)^2+zm(z)-1 \ , \quad  \ .
\ee
which satisfies $m(0)=\sqrt{2/\pi}$.  Observe that $m(z)=-z$ is a solution to (\ref{F8}). 

We show that the function (\ref{C8}) is decreasing.  
To see this we use the identity
\be\label{G8}
 \int_0^\infty z'(z'+z)e^{-z'^2/2-zz'} \ dz' \ = \  \int_0^\infty z'\left(-\frac{d}{dz'}\right) e^{-z'^2/2-zz'} \ dz'       \\
 = \  \int_0^\infty e^{-z'^2/2-zz'} \ dz' \ .
\ee
Evidently  (\ref{G8}) implies that
\be \label{H8}
\langle X_z^2\rangle+z\langle X_z\rangle \ = \ 1 \ .
\ee
Since $\langle X_z\rangle^2< \langle X_z^2\rangle$, we see from (\ref{H8}) that the RHS of (\ref{F8}) is strictly negative, whence $m(\cdot)$ decreases. Note  also that (\ref{H8}) implies the upper bound in (\ref{A8}). 
To obtain the lower bound we recall the well known inequality 
\be \label{I8}
\left[\frac{1}{z}-\frac{1}{z^3}\right] \  < \ e^{z^2/2}\int_z^\infty e^{-z'^2/2} \ dz' <  \frac{1}{z} \  \quad {\rm for \ } z>0. 
\ee
Considering (\ref{F8}) to be a linear equation with inhomogeneous term $m(z)^2-1$, the solution $m(z)$ has the representation
\be \label{J8}
m(z) \ = \ e^{z^2/2}\int_z^\infty [1-m(z')^2]e^{-z'^2/2} \ dz' \ .
\ee
It follows from the upper bound in (\ref{A8}) and (\ref{J8}) that
\be \label{K8}
m(z) \ \ge \  \left(1-\frac{1}{z^2}\right) e^{z^2/2}\int_z^\infty e^{-z'^2/2} \ dz' \ .
\ee
We then obtain the lower bound in (\ref{A8}) from the lower bound in (\ref{I8}) and  (\ref{K8}). 

To obtain the lower bound in (\ref{B8}) we use the fact that  trajectories of the non-autonomous differential equation (\ref{F8}) do not intersect  in $\mathbb{R}^2$, in particular the trajectories $z\ra [m(z),z]$ and  $z\ra [-z,z]$.  To obtain the upper bound we observe that the function $w(z)=\al-z-m(z)$ is a solution to the initial value problem
\be \label{L8}
\frac{dw(z)}{dz} \ = \ a(z)w(z)-b(z) \ ,  \quad w(0)=\al-\sqrt{2/\pi} \ ,
\ee
where the functions $a(\cdot), \ b(\cdot)$ are given by
\be \label{M8}
a(z) \ =  \ \al+m(z) \ , \quad b(z) \ = \ \al[\al-z] \  .
\ee
Evidently $w(z)\ge 0$ for $z<0$ if $w(0)\ge 0$ and $b(z)\ge 0$ for $z<0$. This is the case if $\al=\sqrt{2/\pi}$. 
\end{proof}
\begin{lem}
For any  $\del$ satisfying $0<\del<1$ there exists $c(\del)>0$ depending on $\del$ such that for all $z\in\mathbb{R}$, 
\be \label{N8}
P(m(z)<X_z<(1+\del)m(z)) \ \ge \ c(\del) \ ,  \quad P((1-\del) m(z)<X_z<m(z)) \ \ge \ c(\del) \  .
\ee
Furthermore there exist constants $C,c>0$ such that for all $z\in\mathbb{R}$, 
\be \label{O8}
P(X_z>km(z)) \ \le \ \ Ce^{-ck} \ , \quad k=1,2,...
\ee
The variables $X_z/m(z)$ converge in distribution as $z\ra\infty$ to the exponential variable $\mathcal{X}$ with mean $1$.
\end{lem}
\begin{proof}
Let $\rho_z(z'), \ z'>0,$  be the pdf of the variable $X_z$.  Then
\be \label{P8}
\rho_z(z') \ = \  A(z)\exp\left[-\frac{(z'+z)^2}{2}\right] \ = \ B(z) \exp\left[-z'z-\frac{z'^2}{2}\right] \ , \quad z'>0 \ ,
\ee
where $A(z),B(z)$ depend only on $z$.  For $z<1$ the variable $X_z$ is approximately Gaussian with mean $m(z)$. Hence using the first representation on the RHS of (\ref{P8})  and Lemma A.1. we conclude that (\ref{N8}), (\ref{O8}) hold. For $z\ge 1$ the variable $X_z$ is approximately exponential with mean $m(z)$. Using the second representation on the RHS of (\ref{P8}) we see that (\ref{N8}), (\ref{O8}) hold.  We similarly see that $X_z/m(z)$ converges in distribution to the exponential variable as $z\ra\infty$.  
\end{proof}
\begin{lem}
     Let $f:[0,\infty)\ra\R$ be a continuous non-negative increasing function. Then there is a universal constant $c>0$ such that
\be \label{Q8}
\frac{E[X_zf(X_z))]}{E[f(X_z)]} \ \ge c m(z) \quad {\rm for \ } z\in\mathbb{R} \ .
\ee
If in addition $\lim_{z\ra\infty}f(z)=1$ and $-\log[1-f(\cdot)]$ is a concave function, then there is a universal constant $C$ such that
\be \label{R8}
\frac{E[X_zf(X_z)]}{E[f(X_z)]} \ \le C m(z) \quad {\rm for \ } z\in \mathbb{R} \ .
\ee
\end{lem}
\begin{proof}
The ratio of expectations in (\ref{Q8}) is given by
\be \label{S8}
\int_0^\infty z'f(z') \rho_z(z') \ dz' \Big/\int_0^\infty f(z') \rho_z(z') \ dz' \  ,
\ee
where $\rho_z(\cdot)$ is the pdf of $X_z$. 
From Lemma A.2 we see there is a constant $C_1>0$ such that
\be \label{T8}
\int_0^{m(z)} \rho_z(z') \ dz' \ \le \ C_1\int^\infty_{m(z)} \rho_z(z') \ dz' \ , \quad z\in\R \ .
\ee
It  follows from(\ref{T8}) that the LHS of (\ref{S8}) is bounded below by $m(z)/(C_1+1)$, whence (\ref{Q8}) follows. 

To obtain the upper bound (\ref{R8}) we first observe that the expression on the LHS of  (\ref{R8}) is bounded above by
\be \label{U8}
m(z)+\int_{m(z)}^\infty z'f(z') \rho_z(z') \ dz' \Big/\int_0^\infty f(z') \rho_z(z') \ dz' \ .
\ee
Suppose now that $f(z')\ge1-e^{-1}$ for $z'\ge m(z)$.  Then since $\sup f(\cdot)=1$, it follows from (\ref{N8}) that the expression  (\ref{U8}) is bounded above by $C_2m(z)$ for some constant $C_2$. Hence to complete the proof of (\ref{R8}) we may assume that $f(m(z))\le 1-e^{-1}$.  In that case $f(z')=1-e^{-q(z')}$, where $q(\cdot)$ is a non-negative increasing concave function and $q(m(z))\le1$. From the concavity of $q(\cdot)$ we  have that $q(z')\le \tilde{q}(z')$ for $z'\ge m(z)$, and $q(z')\ge \tilde{q}(z')$ for $m(z)/2<z'<m(z)$, where $\tilde{q}(\cdot)$ is the secant line function
\be \label{V8}
\tilde{q}(z') \ = \ \frac{2}{m(z)}\left[(m(z)-z')q(m(z)/2)+(z'-m(z)/2)q(m(z))\right]\ .
\ee
Hence the second term in (\ref{U8}) is bounded above by
\begin{multline} \label{W8}
\int_{m(z)}^\infty z'[1-e^{-\tilde{q}(z')}] \rho_z(z') \ dz' \Big/\int_{m(z)/2}^{m(z)}  [1-e^{-\tilde{q}(z')}]  \rho_z(z') \ dz' \\
\le \ e\int_{m(z)}^\infty z'\tilde{q}(z') \rho_z(z') \ dz' \Big/\int_{m(z)/2}^{m(z)} \tilde{q}(z')  \rho_z(z') \ dz' \ .
\end{multline}
We see from (\ref{V8}) that
\begin{multline} \label{X8}
\tilde{q}(z') \ \ge \ q(m(z))/2 \quad {\rm for \ } 3m(z)/4\le z'\le m(z) \ , \\
\tilde{q}(z') \ \le \ (2n+1)q(m(z)) \quad {\rm for \ }  z'\le (n+1) m(z) \ , \ n=1,2,...
\end{multline}
It follows from (\ref{W8}), (\ref{X8}) that the second term in (\ref{U8}) is bounded above by
\be \label{Y8}
2e\sum_{n=1}^\infty (2n+1)\int_{nm(z)}^{(n+1)m(z)} z' \rho_z(z') \ dz' \Big/\int_{3m(z)/4}^{m(z)}  \rho_z(z') \ dz' \ .
\ee
We conclude from (\ref{N8}), (\ref{O8})  that the expression (\ref{Y8}) is bounded above by $Cm(z), \ z\in\R,$ for some constant $C$.
\end{proof}

\end{document}